\documentclass[12pt]{article}

\usepackage{diagbox}

\usepackage{amsfonts}
\usepackage{amssymb}
\usepackage{mathrsfs}
\usepackage{srcltx}
\usepackage{amsmath}
\usepackage{amsthm}
\usepackage{amstext}
\usepackage{amsopn}
\usepackage{graphicx}
\usepackage{bm}
\usepackage{color}
\newtheorem{theorem}{Theorem}[section]
\newtheorem{lemma}[theorem]{Lemma}
\newtheorem{proposition}[theorem]{Proposition}

\theoremstyle{definition}

\theoremstyle{remark}

\numberwithin{equation}{section}

\newcommand{\ba}{\begin{array}}
\newcommand{\ea}{\end{array}}
\newcommand{\f}{\frac}

\newcommand{\la}{\lambda}

\newcommand{\ds}{\displaystyle}

\begin{document}
\date{}
\title{ \bf\large{Hopf bifurcation and periodic solutions in a coupled Brusselator model of chemical reactions}}
\author{Yihuan Sun\textsuperscript{1},\ \ Shanshan Chen\textsuperscript{2}\footnote{Corresponding Author, Email: chenss@hit.edu.cn}\ \
 \\
{\small \textsuperscript{1} School of Mathematics, Harbin Institute of Technology,\hfill{\ }}\\
\ \ {\small  Harbin, Heilongjiang, 150001, P.R.China.\hfill{\ }}\\
{\small \textsuperscript{2} Department of Mathematics, Harbin Institute of Technology,\hfill{\ }}\\
\ \ {\small Weihai, Shandong, 264209, P.R.China.\hfill{\ }}\\
}
\maketitle

\begin{abstract}

In this paper, we consider a coupled Brusselator model of chemical reactions, for which no symmetry for the coupling
matrices is assumed.
We show that the model can undergoes a Hopf bifurcation, and consequently periodic solutions can arise when the dispersal rates are large. Moreover, the effect of the coupling matrices on the Hopf bifurcation value is considered for a special case.

\end{abstract}

\section{Introduction}

Nonlinear oscillations often occur in many chemical reactions and physical processes, \cite{Field1974,Kuramoto,Lefever1971}. For example,
sustained oscillations in Brusselator model were studied in \cite{Lefever1971}.
Brusselator model, proposed by Prigogine and Lefever \cite{1968}, describes a set of chemical reactions as follows:
\begin{equation}\label{reaction}
A\rightarrow X,\;\;B+X\rightarrow Y+C,\;\;2X+Y\rightarrow 3X,\;\;X\rightarrow D,
\end{equation}
where $A$ and $B$ are the concentrations of the initial substances, $X$ and $Y$ are the concentrations of the intermediate reactants, and $C$ and $D$ are the concentrations of the final products.
If the
concentrations $A$ and $B$ depend only on space, these chemical reactions can be modelled by the following reaction-diffusion model: (see \cite{Lefever1971,1968})
\begin{equation}\label{rdbru}
\begin{cases}
X_t= d_1\Delta X + A(x) - (B(x) + 1)X +X^2Y,&x\in\Omega,\;t>0,\\
Y_t = d_2\Delta Y  + B(x)X - X^2Y, &x\in\Omega,\;t>0.
\end{cases}
\end{equation}
Here $d_1,d_2>0$ are the dispersal rates.
If $A$ and $B$ are spatially homogeneous, then model \eqref{rdbru} admits a unique constant steady state for the homogeneous Neumann boundary condition.
One can refer to \cite{Choi,Guo2011,LiB2008,Liyan2016,LiaoMX2016,Wittenberg1997,Yan2020} and references therein for the results on steady state
and Hopf bifurcations near this constant steady state. The global bifurcation theory and some other methods were used to show the existence of non-constant steady states for a wide range of parameters, see, e.g.,
\cite{Brown1995,JiaLiWu2016,MaMJ2014,PengR2005,Peng2009}.
Moreover, the effect of advection was considered in \cite{Andresen1999,Kuznetsov2002}, and the diffusion term in \eqref{rdbru} was replaced by
the following form:
\begin{equation*}
d_i\ds\f{\partial^2}{\partial x^2}\to d_i\ds\f{\partial^2}{\partial x^2}-c_i\ds\f{\partial}{\partial x}
\end{equation*}
for the case that $\Omega$ is one dimension.

If chemical reactions \eqref{reaction} take places in $n$-boxes,
then model \eqref{rdbru} takes the following discrete form:
\begin{equation}\label{m1}
\begin{cases}
\ds\frac{d x_j}{d t}= d_1\sum_{k=1}^np_{jk}x_k + a_j - (\beta b_j + 1)x_j + x_j^2y_j, &j=1,\dots,n,\;\;t>0,\\
\ds\frac{d y_j}{d t} = d_2\sum_{k=1}^n q_{jk}y_k  + \beta b_jx_j - x_j^2y_j,&j=1,\dots,n,\;\;t>0, \\
\bm x(0)=\bm x_0\ge(\not\equiv)\bm0,\;\bm y(0)=\bm y_0\ge(\not\equiv)\bm0.
\end{cases}
\end{equation}
Here $n \geq 2$ is the number of boxes in chemical reactions;
$\bm x=(x_1,\dots,x_n)^T$ and $\bm y=(y_1,\dots,y_n)^T$, where
$x_j$ and $y_j$ denote the concentrations of $X$ and $Y$ in box $i$ at time $t$, respectively; the nonlinear term ${x_j^2}y_j$ describes the autocatalytic step in box $j$;
and ${a_j}>0$ and $\beta{b_j}>0$ denote the input concentrations of initial substances $A$ and $B$ in box $j$, respectively. We remark that
$\beta$ introduced here is the scaling parameter, and we choose it as the bifurcation parameter.
Moreover, $\left(p_{jk}\right)$ and $\left(q_{jk}\right)$ are the coupling matrices, where
$p_{jk}, q_{jk} (j\ne k)$ describe the rates of movement from box $k$ to box $j$ for the two reactants, respectively, and $p_{jj}(=-\sum_{k\ne j} p_{kj})$ and $q_{jj}(=-\sum_{k\ne j} q_{kj})$ denote the rates of leaving box $j$
 for $j=1,\dots,n$.

Model \eqref{m1} was first considered by Prigogine and Lefever in \cite{1968} for the case of two boxes ($n=2$), where the coupling matrices $\left(p_{jk}\right)$ and $\left(q_{jk}\right)$ were assumed to be symmetric, and the two boxes were identical (that is,
$a_1=a_2=a$ and $b_1=b_2=1$). In this case, model \eqref{m1} admits a unique space-independent steady state:
$$x_i=a,\;\; y_i=\beta/a\;\;\text{for}\;\;i=1,2,$$
and the existence of space-dependent steady state was showed in \cite{1968}.
Model \eqref{m1} with symmetric coupling matrices and identical boxes were also considered in \cite{Lengyel1991,LiShi2007,Schreiber1986,ShiLi2007}, where the steady state and Hopf
bifurcation were studied in \cite{Lengyel1991,Schreiber1986}. One can also refer to \cite{Duan2019,Petit2016,Tian2019,Yu2001} and references therein for the steady state and Hopf bifurcations of
other coupled models with symmetric coupling matrices.

It is well-known that the symmetric coupling matrices could mimic random diffusion, and the asymmetric case could mimic advective movements in the fluid.
If the coupling matrices $\left(p_{jk}\right)$ and $\left(q_{jk}\right)$ are asymmetric, the steady states of model \eqref{m1}
are space-dependent even when the boxes are all identical.
Consequently, we cannot obtain the explicit expression of the steady states, which brings difficulties in analyzing the Hopf bifurcation.
In this paper, we aim to solve this problem and analyze the Hopf bifurcation of model \eqref{m1}.

Throughout the paper, we impose the following two assumptions:
\begin{enumerate}
  \item[$(\bf A_1)$] The coupling matrices $P:=(p_{jk})$ and $Q:=(q_{jk})$ are irreducible and essentially nonnegative;
  \item[$(\bf A_2)$] $\ds\frac{d_2}{d_1}:=\theta>0$.
\end{enumerate}
Here we remark that real matrices with nonnegative off-diagonal elements are called essentially nonnegative.
It follows from $(\bf A_1)$ and Perron-Frobenius theorem that $s(P)=s(Q)=0$, where $s(P)$ and $s(Q)$ are spectrum bounds of $P$ and $Q$, respectively.
Clearly, $s(P)$ and $s(Q)$ are also simple eigenvalues of $P$ and $Q$
with strongly positive eigenvectors ${\bm \xi}$ and ${\bm \eta}$, respectively, where
\begin{equation}\label{xi}
\begin{split}
&\bm{\xi}=(\xi_1,\cdots, \xi_n)^T\gg \bm 0,\;\;\text{and}\;\; \sum_{j=1}^{n} \xi_j = 1,\\
&\bm{\eta}=(\eta_1,\cdots, \eta_n)^T\gg \bm 0,\;\;\text{and}\;\; \sum_{j=1}^{n} \eta_j = 1.\\
\end{split}
\end{equation}
Here $(x_1,\cdots, x_n)^T\gg \bm 0$ represents that $x_j>0$ for $j=1,\dots,n$.
Assumption $(\bf A_2)$ is a mathematically technical condition, and it means that the dispersal rates of the two reactants are proportional.
Then letting $\tilde t=d_1 t$, denoting $\la= 1/ d_1$, and dropping the tilde sign, model \eqref{m1} can be transformed to the following equivalent model:
\begin{equation}\label{m2}
\begin{cases}
\ds\frac{d x_j}{d t}= \sum_{k=1}^np_{jk}x_k +\la \left[ a_j - (\beta b_j + 1)x_j + x_j^2y_j\right], &j=1,\dots,n,\;\;t>0,\\
\ds\frac{d y_j}{d t} = \theta\sum_{k=1}^n q_{jk}y_k  + \la \left(\beta b_jx_j - x_j^2y_j\right),&j=1,\dots,n,\;\;t>0, \\
\bm x(0)=\bm x_0\ge(\not\equiv)\bm0,\;\bm y(0)=\bm y_0\ge(\not\equiv)\bm0,
\end{cases}
\end{equation}
where $\theta$ is defined in assumption $(\bf A_2)$.


We point out that the method in this paper is motivated by \cite{Busenberg96}, where the steady state of the model is space-dependent. One can also refer to \cite{AnWangWang,Chen18,Chen12,Guo15,GuoYanJDE,Hu11,JinYuan,LiDai,Su09,Yan10,YanLiDCDS} on Hopf bifurcations near this type of steady state for delayed reaction-diffusion equations and delayed differential equations.
We remark that, for the model in \cite{Busenberg96},  the steady state does not depends on bifurcation parameter. But for patch models, the steady states always depend on the bifurcation parameter (see $\beta$ in model \eqref{m2}), which brings some technical hurdles in analyzing Hopf bifurcations. Therefore, we need to improve the method in \cite{Busenberg96} here.

For simplicity, we use the following notations. We denote the complexification of a real linear space $\mathcal Z$ to be $\mathcal Z_{\mathbb C}:=\mathcal Z\oplus{\rm i}\mathcal Z=\{x_1 + {\rm i}x_2 |x_1, x_2 \in \mathcal Z\}$, and define the kernel of a linear operator $T $ by $\mathcal N (T)$. For $\mu\in\mathbb C$, we define the real part by $\mathcal Re \mu$.
For the complex valued space $\mathbb C^n$, we choose the standard inner product $\langle \bm u,\bm v    \rangle=\sum_{j=1}^{n}\overline u_jv_j$, and consequently, the norm is defined by
$$
\| \bm u\|_2=\left(\sum_{j=1}^{n} |u_j |^2\right)^{\frac{1}{2}} \;\;\text{for}\;\; \bm u \in \mathbb C^n.
$$
Moreover, for $\bm x=(x_1,\dots,x_n)^T,\bm y=(y_1,\dots,y_n)^T\in\mathbb C^n$, we denote $$(\bm x,\bm y)^T=\left( {\begin{array}{c}
{\bm x}\\
{\bm y}
\end{array}} \right)=(x_1,\dots,x_n,y_1,\dots,y_n)^T.$$

The rest of the paper is organized as follows. In Section 2, we study the existence and uniqueness
of the positive equilibrium for model \eqref{m2} (or respectively, \eqref{m1}). In Section 3, we show the existence of Hopf bifurcation and the stability of  the positive equilibrium for model \eqref{m2} (or respectively, \eqref{m1}) when $\la$ is small. In Section 4, we show the effect of the coupling matrices on the Hopf bifurcation value for a special case. Finally, some numerical simulations are provided to illustrate the theoretical results.
\section{Existence of positive equilibria}

In this section, we consider the existence of positive equilibria of model \eqref{m2}, which satisfy
\begin{equation}\label{equi1}
\begin{cases}
\sum_{k=1}^n p_{jk}x_k + \la \left[a_j-(\beta b_j + 1)x_j+x_j^2y_j\right]=0, & j=1,\dots,n,\\
\theta \sum_{k=1}^n q_{jk}y_k  + \la \left(\beta b_jx_j - x_j^2y_j\right)=0, & j=1,\dots,n.
\end{cases}
\end{equation}
Note that $(\bm x,\bm y)=(c\bm \xi,d\bm\eta)$ solves \eqref{equi1} for all $c,d\in\mathbb R$ when $\la=0$. Therefore, we cannot solve
\eqref{equi1} by the direct application of the implicit function theorem. We need to split the phase space and find an equivalent system of \eqref
{equi1}.
It is well-known that
\begin{equation*}
\mathbb R^n={\rm span}\{\bm\xi\}\oplus  \mathcal M={\rm span}\{\bm\eta\}\oplus  \mathcal M,
\end{equation*}
where $\bm \xi$ and $\bm \eta$ are defined in \eqref{xi}, and
\begin{equation*}
\mathcal M:=\left\{\bm x=(x_1,\dots,x_n)^T\in\mathbb R^n:\sum_{j=1}^n x_j=0\right\}.
\end{equation*}
Letting
\begin{equation}\label{xy}
\begin{split}
&\bm x= c\bm \xi+\bm u,\;\;c\in\mathbb R,\;\bm u\in \mathcal M,\\
&\bm y= r\bm \eta+\bm v,\;\;r\in\mathbb R,\;\bm v\in \mathcal M,\\
\end{split}
\end{equation}
and plugging \eqref{xy} into \eqref{equi1}, we see that $(\bm x,\bm y)$ (defined in \eqref{xy}) is a solution of \eqref{equi1}, if and only if $(c,r, \bm u,\bm v)\in \mathbb R^2\times \mathcal M^2$ solves
\begin{equation}\label{equi2}
\bm F(c,r,\bm u,\bm v,\beta,\la)=(f_1,f_{21},\dots,f_{2n},f_3,f_{41},\dots,f_{4n})^T=\bm 0,
\end{equation}
where
$\bm F(c,r,\bm u,\bm v,\beta,\la):\mathbb R^2\times \mathcal M^2\times \mathbb R^2\to \left(\mathbb R\times \mathcal M\right)^2$,
and
\begin{equation}\label{fi}
\begin{cases}
f_1(c,r,\bm u,\bm v,\beta,\la):=\ds\sum_{j=1}^n\left[a_j-(\beta b_j+1)(c\xi_j+u_j)+(c\xi_j+u_j)^2(r\eta_j+v_j)\right],\\
f_{2j}(c,r,\bm u,\bm v,\beta,\la):=\la\left[a_j-(\beta b_j+1)(c\xi_j+u_j)+(c\xi_j+u_j)^2(r\eta_j+v_j)\right]\\
~~~~~~~~~~~~~~~~~~~~~~~~~~~\ds+\sum_{k=1}^n p_{jk}u_k-\ds\frac{\la}{n} f_1,~~~~~~j=1,\dots,n,\\
f_{3}(c,r,\bm u,\bm v,\beta,\la):=\ds\sum_{j=1}^n\left[\beta b_j(c\xi_j+u_j)-(c\xi_j+u_j)^2(r\eta_j+v_j)\right],\\
f_{4j}(c,r,\bm u,\bm v,\beta,\la):=\la\left[\beta b_j(c\xi_j+u_j)-(c\xi_j+u_j)^2(r\eta_j+v_j)\right]\\
~~~~~~~~~~~~~~~~~~~~~~~~~~~\ds+\theta\sum_{k=1}^n q_{jk}v_k-\ds\frac{\la}{n} f_3,~~~~~~j=1,\dots,n.\\
\end{cases}
\end{equation}

We first solve $\bm F(c,r,\bm u,\bm v,\beta,\la)=\bm 0$ for $\la=0$.
\begin{lemma}\label{l1}
Assume that $\la=0$. For fixed $\beta>0$, $\bm F(c,r,\bm u,\bm v,\beta,\la)=\bm 0$ has a unique solution $(c_0,r_{0\beta},\bm u_0,\bm v_0)\in \mathbb R^2\times \mathcal M^2$, where
\begin{equation}\label{c0rbeta}
c_0 =\sum_{j=1}^n a_j,\;\; r_{0\beta} = \frac{\beta\sum_{j=1}^nb_j\xi_j}{\left(\sum_{j=1}^na_j\right)\left(\sum_{j=1}^n \xi_j^2\eta_j\right)},\;\;\bm u_0=\bm 0,\;\;\bm v_0=\bm 0.
\end{equation}
\end{lemma}
\begin{proof}
Plugging $\la =0$ into $f_{2j}=f_{4j}=0$ for $j=1,\dots,n$, we have $\bm u=\bm u_0=\bm 0$ and $\bm v=\bm v_0=\bm 0$. Then
plugging $\bm u=\bm v=\bm0$ into $f_1=f_3=0$, we have
\begin{equation*}
\sum_{j=1}^n \left[a_j-c(\beta b_j+1)\xi_j+c^2r\xi_j^2\eta_j\right]=0,\;\;\sum_{j=1}^n \left[c\beta b_j \xi_j-c^2r\xi_j^2\eta_j\right]=0,
\end{equation*}
which implies that
\begin{equation*}
c=c_0 =\sum_{j=1}^n a_j,\;\; r=r_{0\beta} = \frac{\beta\sum_{j=1}^nb_j\xi_j}{\left(\sum_{j=1}^na_j\right)\left(\sum_{j=1}^n \xi_j^2\eta_j\right)}.
\end{equation*}
This completes the proof.
\end{proof}
Now we solve \eqref{equi1} (or equivalently, \eqref{equi2}) for $\la>0$.
\begin{theorem}\label{thlocal}
For any fixed $\beta_1>0$, there exists $\delta_{\beta_1}\in(0,\beta_1)$ and a continuously differentiable mapping
$\left(\bm x^{(\la,\beta)},\bm y^{(\la,\beta)}\right):[0,\delta_{\beta_1}]\times [\beta_1-\delta_{\beta_1},\beta_1+\delta_{\beta_1}]\to \mathbb R^n\times \mathbb R^n $ such that
$(\bm x^{(\la,\beta)},\bm y^{(\la,\beta)})$ is the unique positive solution of \eqref{equi1} for $(\la,\beta)\in (0,\delta_{\beta_1}]\times [\beta_1-\delta_{\beta_1},\beta_1+\delta_{\beta_1}]$.
Moreover,
\begin{equation}\label{xyf}
\bm x^{(\la,\beta)}=c^{(\la,\beta)}\bm\xi+\bm u^{(\la,\beta)},\;\; \bm y^{(\la,\beta)}=r^{(\la,\beta)}\bm\eta+\bm v^{(\la,\beta)},
\end{equation}
where
$\left(c^{(\la,\beta)},r^{(\la,\beta)},\bm u^{(\la,\beta)},\bm v^{(\la,\beta)}\right)\in \mathbb R^2\times \mathcal M^2$ solves
Eq. \eqref{equi2} for $(\la,\beta)\in [0,\delta_{\beta_1}]\times [\beta_1-\delta_{\beta_1},\beta_1+\delta_{\beta_1}]$, and
\begin{equation}\label{limcruv}
\left(c^{(0,\beta)},r^{(0,\beta)},\bm u^{(0,\beta)},\bm v^{(0,\beta)}\right)=(c_0,r_{0\beta},\bm 0,\bm 0)\;\;\text{for}\;\;
\beta\in[\beta_1-\delta_{\beta_1},\beta_1+\delta_{\beta_1}],
\end{equation}
with $c_0$ and $r_{0\beta}$ defined in Lemma \ref{l1}.
\end{theorem}
\begin{proof}
We first show the existence.
It follows from Lemma \ref{l1} that $$\bm F(c_0,r_{0\beta_1},\bm 0,\bm 0,\beta_1,0)=\bm 0,$$
where $\bm F$ is defined in \eqref{equi2}. A direct computation implies that
the Fr\'echet derivative of $\bm F$ with respect to $(c,r,\bm u,\bm v)$ at $(c_0,r_{0\beta_1},\bm 0,\bm 0,\beta_1,0)$ is as follows:
\begin{equation*}
\bm G(\tilde c,\tilde r,\bm {\tilde u},\bm {\tilde v})=(g_1,g_{21},\dots,g_{2n},g_3,g_{41},\dots,g_{4n})^T,
\end{equation*}
where $\tilde c,\tilde r\in\mathbb R$, $\bm {\tilde u},\bm {\tilde v} \in\mathcal M$, and
\begin{equation*}
\begin{cases}
\ds g_1(\tilde c,\tilde r,\bm {\tilde u},\bm {\tilde v}):=\sum_{j=1}^n\left[ { \left(2 c_0 r_{0\beta_1} \xi_j \eta_j - \beta_1 b_j - 1\right){(\tilde c \xi_j+\tilde u_j)} + {(c_0 \xi_j)}^2 }{(\tilde r \eta_j+\tilde v_j)}\right],\\
\ds g_{2j} (\tilde c,\tilde r,\bm {\tilde u},\bm {\tilde v}):=\sum_{k= 1}^n {p_{jk}\tilde u_k},~~~~~~j=1,\dots,n,\\
\ds g_3(\tilde c,\tilde r,\bm {\tilde u},\bm {\tilde v}):=\sum_{j=1}^n \left[ { \left(\beta_1 b_j  - 2 c_0 r_{0\beta_1} \xi_j \eta_j \right){(\tilde c \xi_j+\tilde u_j)} - {(c_0 \xi_j)}^2 }{(\tilde r \eta_j+\tilde v_j)}\right],\\
\ds g_{4j} (\tilde c,\tilde r,\bm {\tilde u},\bm {\tilde v}):=\theta \sum_{k = 1}^n {q_{jk}\tilde v_k},~~~~~~j=1,\dots,n.\\
\end{cases}
\end{equation*}
If $\bm G(\tilde c,\tilde r,\bm {\tilde u},\bm {\tilde v})=\bm 0$, then $\bm {\tilde u}=\bm 0$ and $\bm {\tilde v}=\bm 0$. Plugging $\bm {\tilde u}=\bm {\tilde v}=\bm 0$ into $g_1=g_3=0$, we have
\begin{equation*}
\left( \begin{array}{cc}
\sum_{j=1}^n \left(2 c_0 r_{0\beta_1} \xi_j \eta_j - \beta_1 b_j - 1\right) \xi_j  &\sum_{j=1}^n {(c_0 \xi_j)}^2 \eta_j\\
\sum_{j=1}^n \left(\beta_1 b_j  - 2 c_0 r_{0\beta_1} \xi_j \eta_j \right) \xi_j   &- \sum_{j=1}^n {(c_0 \xi_j)}^2 \eta_j
\end{array} \right)\left( {\begin{array}{*{20}{c}}
\tilde c\\
\tilde r
\end{array}} \right) = \left( {\begin{array}{*{20}{c}}
0\\
0
\end{array}} \right).
\end{equation*}
Noticing that
\begin{equation*}
\left| \begin{array}{cc}
\sum_{j=1}^n \left(2 c_0 r_{0\beta_1} \xi_j \eta_j - \beta_1 b_j - 1\right) \xi_j &\sum_{j=1}^n {(c_0 \xi_j)}^2 \eta_j\\
\sum_{j=1}^n \left(\beta_1 b_j  - 2 c_0 r_{0\beta_1} \xi_j \eta_j \right) \xi_j  & - \sum_{j=1}^n {(c_0 \xi_j)}^2 \eta_j
\end{array} \right| \ne 0,
\end{equation*}
we obtain that $\tilde c=0$ and $\tilde r=0$. Therefore, $\bm G$ is bijective.

It follows from the implicit function theorem that there exist $ \delta_{\beta_1} \in(0,\beta_1)$
and a continuously differentiable mapping
\begin{equation*}
(\la ,\beta)\in [0, \delta_{\beta_1}] \times [\beta_1-\delta_{\beta_1}, \beta_1+\delta_{\beta_1}]   \mapsto \left(c^{(\la ,\beta)},r^{(\la ,\beta)},\bm u^{(\la ,\beta)},\bm v^{(\la ,\beta)}\right)\in \mathbb R^2\times \mathcal M^2
\end{equation*}
such that $\bm F\left(c^{(\la ,\beta)},r^{(\la,\beta)},\bm u^{(\la ,\beta)},\bm v^{(\la ,\beta)},\beta,\la\right)=\bm 0$ and
\begin{equation*}
\left(c^{(0 ,\beta_1)},r^{(0,\beta_1)},\bm u^{(0,\beta_1)},\bm v^{(0,\beta_1)}\right)=(c_0,r_{0\beta_1},\bm 0,\bm 0).
\end{equation*}
Then we can choose $\delta_{\beta_1}>0$ (sufficiently small) such that
$\left(\bm x^{(\la,\beta)},\bm y^{(\la,\beta)}\right)$ (defined in \eqref{xyf}) is a positive solution of \eqref{equi1} for $(\la,\beta)\in (0, \delta_{\beta_1}] \times [\beta_1-\delta_{\beta_1}, \beta_1+\delta_{\beta_1}]$.

Now, we show the uniqueness. From the implicit function theorem, we only need to verify that if
$\left(\bm {\tilde x}^{(\la,\beta)},\bm {\tilde y}^{(\la,\beta)}\right)$ is a positive solution of \eqref{equi1},
where
\begin{equation*}
\bm{\tilde  x}^{(\la ,\beta)}=\tilde c^{(\la ,\beta)}\bm \xi+\bm {\tilde u}^{(\la ,\beta)},\;
\bm{\tilde  y}^{(\la ,\beta)}=\tilde  r^{(\la ,\beta)}\bm \eta+\bm{\tilde  v}^{(\la ,\beta)},\;\tilde c^{(\la ,\beta)},\tilde  r^{(\la ,\beta)}\in\mathbb{R},\;\bm {\tilde u}^{(\la ,\beta)},\bm{\tilde  v}^{(\la ,\beta)}\in\mathcal M,
\end{equation*}
then $\left(\tilde c^{(\la ,\beta)},\tilde r^{(\la ,\beta)},\bm{\tilde u}^{(\la ,\beta)},\bm {\tilde v}^{(\la ,\beta)}\right) \to (c_0, r_{0\beta_1}, \bm 0,\bm 0 )$ as $(\la ,\beta) \to (0,\beta_1)$. Substituting $(c,r,\bm u,\bm v)=\left(\tilde c^{(\la ,\beta)},\tilde r^{(\la ,\beta)},\bm{\tilde u}^{(\la ,\beta)},\bm {\tilde v}^{(\la ,\beta)}\right)$ into \eqref{equi2}, we see from the first and third equation of \eqref{fi} that
\begin{equation*}
\tilde c^{(\la ,\beta)}=\sum_{j=1}^n {\tilde x_j^{(\la ,\beta)}}= \sum_{j=1}^n a_j,
\end{equation*}
which implies that $\bm{\tilde  u}^{(\la ,\beta)}$ is bounded in $\mathbb R^n$. Then, up to a subsequence, we assume that
\begin{equation*}
\lim_{(\la,\beta)\to (0,\beta_1)}\bm{\tilde u}^{(\la, \beta)}= \bm u^*\in \mathcal M.
 \end{equation*}
From the third equation of \eqref{fi}, we have
\begin{equation}\label{multi}
\sum_{j=1}^n \beta b_j \tilde x_j^{(\la ,\beta)}= \sum_{j=1}^n  \left(\tilde x_j^{(\la ,\beta)}\right)^2 \tilde y_j^{(\la ,\beta)},
 \end{equation}
 which implies that $\left\{\left(\tilde x_j^{(\la ,\beta)}\right)^2\tilde y_j^{(\la ,\beta)}\right\}_{j=1}^n$ are bounded.
 Taking the limit of $$f_{2j}\left(\tilde c^{(\la, \beta)},\tilde r^{(\la, \beta)},\bm{\tilde  u}^{(\la, \beta)},\bm {\tilde v}^{(\la, \beta)},\beta,\la\right)=0 \;\;\text{for}\;\;j=1,\dots,n,$$ as $(\la,\beta)\to(0,\beta_1)$, we see
 that $P\bm u^*=\bm 0$, which yields $\bm u^*=\bm 0$. Therefore,
 \begin{equation}\label{limuc}
\lim_{(\la,\beta)\to(0,\beta_1)}\bm {\tilde u}^{(\la, \beta)} =\bm 0\;\;\;\;\text{and}\;\;\lim_{(\la,\beta)\to(0,\beta_1)}\tilde c^{(\la, \beta)} =  c_0.
 \end{equation}

It follows from \eqref{multi} and \eqref{limuc} that
$\bm{\tilde y}^{(\la ,\beta)}$ is also bounded in $\mathbb R^n$. Then, up to a subsequence, we assume that
\begin{equation*}
\lim_{(\la,\beta)\to(0,\beta_1)}\tilde r^{(\la ,\beta)}=r^*\ge0\;\;\;\;\text{and}\;\lim_{(\la,\beta)\to(0,\beta_1)}\bm {\tilde v}^{(\la,\beta)}=\bm v^*\in\mathcal M.
\end{equation*}
 Taking the limit of
 \begin{equation*}
f_{4j}\left(\tilde c^{(\la, \beta)},\tilde r^{(\la, \beta)},\bm{\tilde u}^{(\la, \beta)},\bm{\tilde v}^{(\la, \beta)},\beta,\la\right)=0 \;\;\text{for}\;\;j=1,\dots,n,
 \end{equation*}
 as $(\la,\beta)\to(0,\beta_1)$, we see
 that $Q\bm v^*=\bm 0$, which yields $\bm v^*=\bm 0$.
Consequently, taking the limit of
$$f_3\left(\tilde c^{(\la, \beta)},\tilde r^{(\la, \beta)},\bm{\tilde u}^{(\la, \beta)},\bm{\tilde v}^{(\la, \beta)},\beta,\la\right)=0$$
as $(\la,\beta)\to(0,\beta_1)$, we have $r^*=r_{0\beta_1}$.
 Therefore,
 \begin{equation*}
\lim_{(\la,\beta)\to(0,\beta_1)}\bm {\tilde v}^{(\la, \beta)}= \bm 0\;\;\;\;\text{and}\;\;\lim_{(\la,\beta)\to(0,\beta_1)}\tilde r^{(\la, \beta)} =  r_{0\beta_1}.
 \end{equation*}

Finally, we need to show that \eqref{limcruv} holds. We can use the similar arguments as in the proof of the uniqueness, and here
we omit the proof.
\end{proof}
In the above Theorem \ref{thlocal}, we solve \eqref{equi1} when $\beta$ is in a small neighborhood of a given positive constant $\beta_1$. In the following, we will consider the solution of \eqref{equi1} for a wider range of $\beta$.
\begin{theorem}\label{thglobal}
Let $\mathcal B:=[\epsilon,1/\epsilon]$, where $\epsilon>0$ is sufficiently small. Then there exists $\delta_{\epsilon} >0$ and a continuously differentiable mapping $\left(\bm x^{(\la,\beta)},\bm y^{(\la,\beta)}\right):[0,\delta_\epsilon]\times \mathcal B\to \mathbb R^n\times \mathbb R^n $ such that
$(\bm x^{(\la,\beta)},\bm y^{(\la,\beta)})$ is the unique positive solution of
 \eqref{equi1} for $(\la,\beta)\in (0,\delta_\epsilon]\times \mathcal B$.
Moreover,
\begin{equation}\label{xyf2}
\bm x^{(\la,\beta)}=c^{(\la,\beta)}\bm\xi+\bm u^{(\la,\beta)},\;\; \bm y^{(\la,\beta)}=r^{(\la,\beta)}\bm\eta+\bm v^{(\la,\beta)},
\end{equation}
where
$\left(c^{(\la,\beta)},r^{(\la,\beta)},\bm u^{(\la,\beta)},\bm v^{(\la,\beta)}\right)\in \mathbb R^2\times \mathcal M^2$ solves Eq. \eqref{equi2} for $(\la,\beta)\in [0,\delta_\epsilon]\times \mathcal B$,
and
\begin{equation}\label{limcruv3}
\left(c^{(0,\beta)},r^{(0,\beta)},\bm u^{(0,\beta)},\bm v^{(0,\beta)}\right)=(c_0,r_{0\beta},\bm 0,\bm 0) \;\;\text{for}\;\;\beta\in\mathcal B,
\end{equation}
with $c_0$ and $r_{0\beta}$ defined in Lemma \ref{l1}.
\end{theorem}
\begin{proof}
It follows from Theorem \ref{thlocal}, for any $\tilde \beta\in\mathcal B$, there exists $\delta_{\tilde \beta}\in(0,\tilde \beta)$ such that, for $(\la,\beta)\in (0,\delta_{\tilde \beta}]\times [\tilde\beta-\delta_{\tilde \beta},\tilde\beta+\delta_{\tilde \beta}]$,
 \eqref{equi1} admits a unique positive solution
$(\bm x^{(\la,\beta)},\bm y^{(\la,\beta)})$,
where $\bm x^{(\la,\beta)}$ and $\bm y^{(\la,\beta)}$ are defined in \eqref{xyf} and continuously differentiable for $(\la,\beta)\in [0,\delta_{\tilde \beta}]\times [\tilde\beta-\delta_{\tilde \beta},\tilde\beta+\delta_{\tilde \beta}]$.
Clearly,
$$\mathcal B \subseteq \bigcup_{\tilde \beta\in\mathcal B}  {\left( {\tilde \beta - \delta_{\tilde \beta},\tilde \beta +\delta_{\tilde \beta}} \right)}. $$
Noticing that $\mathcal B$ is compact, we see that there exist finite open intervals (see Fig. \ref{fig1}), denoted by $\left( {\tilde \beta_l} - \delta_{\tilde \beta_l},{\tilde \beta_l} +\delta_{\tilde \beta_l} \right)$ for $l=1,\dots,s$, such that

$$\mathcal B \subseteq \bigcup_{l=1}^s  {\left( {\tilde \beta_l} - \delta_{\tilde \beta_l},{\tilde \beta_l} +\delta_{\tilde \beta_l} \right)}.$$
Choose $ \delta_\epsilon= \min_{1 \le l \le s} \delta _{\tilde \beta_l}$. Then, for $(\la,\beta)\in (0,\delta_\epsilon]\times \mathcal B$,
 \eqref{equi1} admits a unique positive solution $(\bm x^{(\la,\beta)},\bm y^{(\la,\beta)})$, which is defined in \eqref{xyf2} and continuously differentiable for $(\la,\beta)\in [0,\delta_\epsilon]\times \mathcal B$.
Eq. \eqref{limcruv3} can be proved by using the similar arguments as in the proof of Theorem \ref{thlocal}, and here we omit the proof.
\end{proof}
\begin{figure}[htbp]
\centering
\setlength{\abovecaptionskip}{-1cm}
\includegraphics[width=0.5\textwidth]{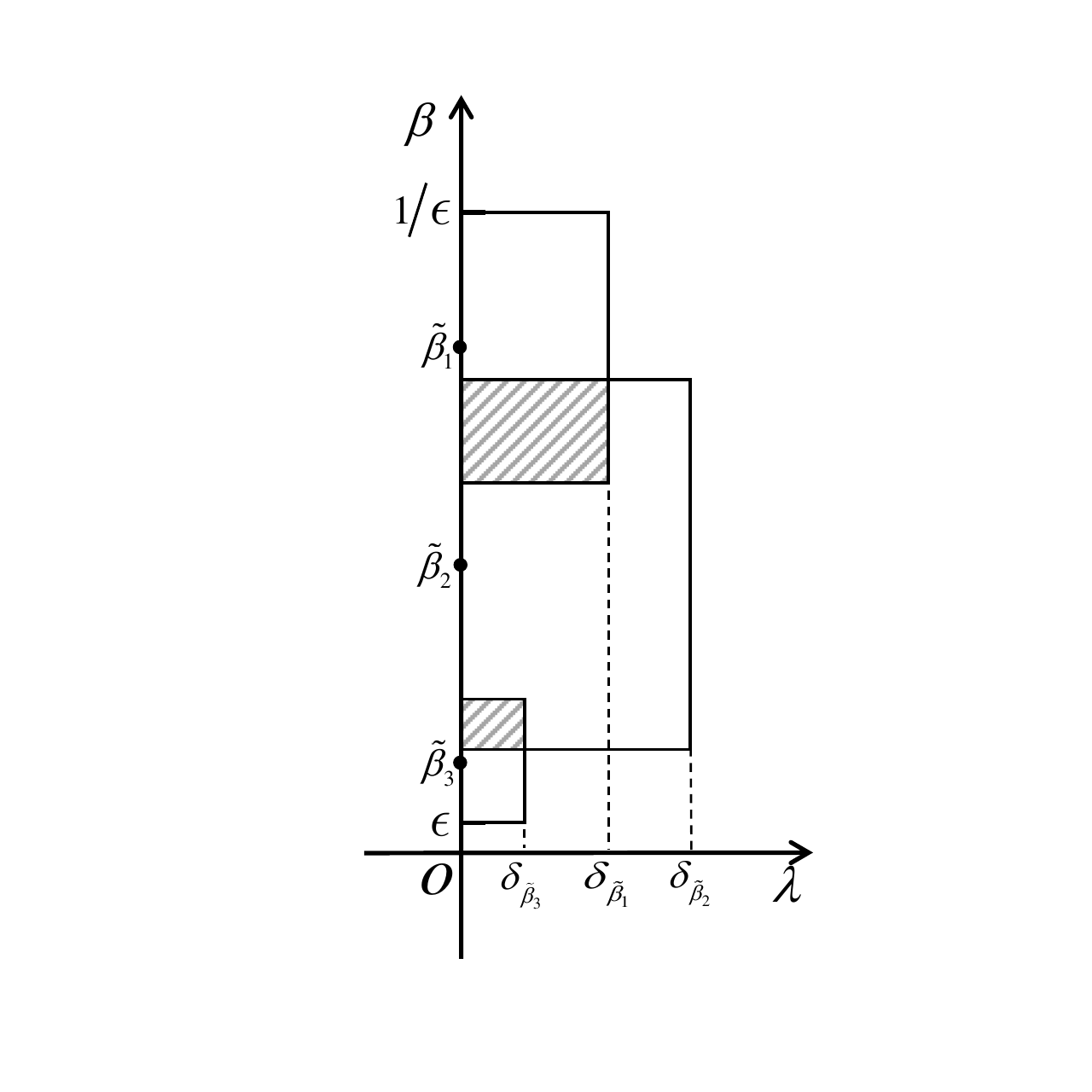}
\caption{Finite open intervals covering $\mathcal B$. \label{fig1}}
\end{figure}

\section{Stability and Hopf bifurcation}

Throughout this section, we assume that
$\beta\in\mathcal B$, where $\mathcal B$ is defined in Theorem \ref{thglobal}.
 It follows from Theorem \ref{thglobal} that there exists $\delta_{\epsilon} >0$ such that, for $(\la,\beta)\in (0,\delta_\epsilon]\times \mathcal B$,
 \eqref{m2} admits a unique positive equilibrium
$(\bm x^{(\la,\beta)},\bm y^{(\la,\beta)})$,
where $\bm x^{(\la,\beta)}$ and $\bm y^{(\la,\beta)}$ are defined in \eqref{xyf2}.
Here we use $\beta$ as the bifurcation parameter, and we will show that
there exists a Hopf bifurcation curve $\beta=\beta_\la$ when $\la$ is small, see Fig. \ref{fig2}.

Linearizing model \eqref{m2} at $(\bm x^{(\la ,\beta)}, \bm y^{(\la ,\beta)})$, we obtain
\begin{equation*}
\begin{cases}
\ds\frac{d \tilde x_j}{d t}= \sum_{k=1}^np_{jk}\tilde x_k +\la \left[ M_{1j}^{(\la,\beta)}\tilde x_j +M_{2j}^{(\la,\beta)} \tilde y_j \right], ~~~~j=1,\dots,n,\;t>0,\\
\ds\frac{d \tilde y_j}{d t} = \theta\sum_{k=1}^n q_{jk} \tilde y_k  + \la \left[ M_{3j}^{(\la,\beta)}\tilde x_j -M_{2j} ^{(\la,\beta)}\tilde y_j\right],~~~~j=1,\dots,n,\;t>0, \\
\end{cases}
\end{equation*}
where
\begin{equation}\label{Mi}
M_{1j}^{(\la,\beta)}= 2x_j^{(\la ,\beta)}y_j^{(\la ,\beta)} - \beta b_j - 1,\;\;M_{2j}^{(\la,\beta)}={\left( x_j^{(\la ,\beta)} \right)}^2,\;\;
M_{3j}^{(\la,\beta)}=\beta b_j - 2x_j^{(\la ,\beta)}y_j^{(\la ,\beta)}.
\end{equation}
Let\begin{equation}\label{ph1}
 A_\beta(\la) : = \left( {\begin{array}{*{20}{c}}
P&0\\
0&\theta Q
\end{array}} \right) + \la\left( {\begin{array}{*{20}{c}}
M_1&M_2\\
M_3&- M_2
\end{array}} \right),
\end{equation}
where $ M_1={\rm diag}\left(M_{1j}^{(\la,\beta)}\right),M_2={\rm diag}\left(M_{2j}^{(\la,\beta)}\right)$ and $M_3={\rm diag}\left(M_{3j}^{(\la,\beta)}\right)$. Then, $\mu \in \mathbb C$ is an eigenvalue of $A_{\beta}(\la)$ if there exists $ (\bm\varphi,\bm \psi)^T(\ne\bm 0)\in \mathbb C^{2n}$ such that
\begin{equation}\label{m4}
\begin{cases}
\ds\mu \varphi _j= \sum_{k=1}^np_{jk}\varphi _k+\la \left[ M_{1j}^{(\la,\beta)} \varphi _j + M_{2j}^{(\la,\beta)} \psi _j \right], &j=1,\dots,n,\\
\ds\mu \psi _j = \theta\sum_{k=1}^n q_{jk} \psi _k + \la \left[ M_{3j}^{(\la,\beta)} \varphi _j -M_{2j}^{(\la,\beta)} \psi _j \right],&j=1,\dots,n. \\
\end{cases}
\end{equation}

For further application, we first give a \textit{priori} estimates for solutions of eigenvalue problem \eqref{m4}.

\begin{lemma} \label{l3}
Let $\mathcal B$ and $\delta_\epsilon$ are defined in Theorem \ref{thglobal}.
Assume that, for $\la\in(0,\delta_\epsilon]$, $(\mu_\la ,\beta_\la,\bm \varphi_\la,\bm \psi_\la)$ solves \eqref{m4} with $\mathcal Re \mu_\la\ge0$, $(\bm\varphi_\la,\bm \psi_\la)^T(\ne\bm 0)\in \mathbb C^{2n}$ and $\beta_\la\in \mathcal B$. Then there exists $\la_1\in(0,\delta_\epsilon)$ such that $\left|\mu_\la/\la\right|$ is bounded for $\la \in (0, {\la}_1]$.
\end{lemma}

\begin{proof}
Ignoring a scalar
factor, we assume that $\|\bm \varphi_\la\|_2^{2}+\|\bm \psi_\la\|_2^{2}=\|\bm \xi\|_2^2+\|\bm \eta\|_2^2$, where $\bm \xi$ and $\bm \eta$ are defined in \eqref{xi}.
Substituting $(\mu,\beta,\bm \varphi,\bm \psi)=(\mu_\la ,\beta_\la,\bm \varphi_\la,\bm \psi_\la)$ into \eqref{m4}, we have
\begin{equation}\label{eigen1}
\begin{cases}
\ds\mu_\la \varphi _{\la j}= \sum_{k=1}^np_{jk}\varphi_{\la k} +\la \left[ M_{1j}^{(\la,\beta_\la)} \varphi _{\la j} + M_{2j}^{(\la,\beta_\la)}\psi _{\la j} \right],&j=1,\dots,n, \\
\ds\mu_\la \psi _{\la j}= \theta\sum_{k=1}^n q_{jk} \psi_{\la k}  + \la \left[ M_{3j}^{(\la,\beta_\la)}\varphi_{\la j} -M_{2j}^{(\la,\beta_\la)}\psi_{\la j} \right],&j=1,\dots,n,\\
\end{cases}
\end{equation}
where $M_{lj}^{(\la,\beta)}$ are defined in \eqref{Mi} for $l=1,2,3$.
Multiplying the first and second equation of \eqref{eigen1} by $ \overline\varphi _{\la j}$ and $ \overline\psi_{\la j}$, respectively, and
summing these equations over all $j$ yield
\begin{equation}\label{infesti}
\begin{split}
&\mu_\la\left(\|\bm \varphi_\la\|_2^{2}+\|\bm \psi_\la\|_2^{2}\right)\\
=&\sum_{j=1}^n\sum_{k=1}^np_{jk}\overline\varphi_{\la j}\varphi_{\la k}+\sum_{j=1}^n\sum_{k=1}^nq_{jk}\overline\psi_{\la j}\psi_{\la k}\\
&+\la \sum_{j=1}^n\left( M_{1j}^{(\la,\beta_\la)}\varphi _{\la j} + M_{2j}^{(\la,\beta_\la)} \psi _{\la j} \right)\overline\varphi_{\la j}+\la \sum_{j=1}^n\left( M_{3j}^{(\la,\beta_\la)}\varphi_{\la j} -M_{2j}^{(\la,\beta_\la)}\psi_{\la j} \right)\overline\psi_{\la j}.
\end{split}
\end{equation}
Note that there exists a positive constant $M^*$ such that
\begin{equation}\label{esti}
\left|M_{1j}^{(\la,\beta_\la)}\right|,\;\left|M_{2j}^{(\la,\beta_\la)}\right|,\;\left|M_{3j}^{(\la,\beta_\la)}\right|\le M^*\;\;\text{for}\;\;\la\in[0,\delta_\epsilon] \;\;\text{and}\;\; j=1,\dots,n.
\end{equation}
This, combined with \eqref{infesti}, implies that $\mu_\la$ is bounded for $\la\in[0,\delta_\epsilon]$.

Now we claim that there exists $\la_1\in(0,\delta_\epsilon)$ such that $\left|\mu_\la/\la\right|$ is bounded for $\la \in (0, {\la}_1]$, where $\delta_\epsilon$ is defined in Theorem \ref{thglobal}.
Note that $\|\bm \varphi_\la\|_2^{2}+\|\bm \psi_\la\|_2^{2}=\|\bm \xi\|_2^2+\|\bm \eta\|_2^2$, $\mathcal Re \mu_\la\ge0$, and $\mu_\la$ is bounded. If the claim is not true, then there exists a sequence $\{\la_l\}_{l=1}^\infty$ such that
$\lim_{l\to\infty}\la_l=0$, $\lim_{l\to\infty}|\mu_{\la_l}/\la_l|=\infty$, $\lim_{l\to\infty}\mu_{\la_l}=\gamma$ with $\mathcal{R}e \gamma \ge 0$, and $\lim_{l\to\infty}\bm{\varphi}_{\la_l}=\bm\varphi_*$ and $\lim_{l\to\infty}\bm{\psi}_{\la_l}=\bm\psi_*$ with $\|\bm{\varphi}_*\|_2^2+\|\bm{\psi}_*\|_2^2=\|\bm \xi\|_2^2+\|\bm \eta\|_2^2$.
Substituting $\la=\la_l$ into \eqref{eigen1} and
taking the limits of the two equations as $l\to\infty$, we have
\begin{equation}\label{limph}
P\bm{\varphi}_*-\gamma\bm{\varphi}_*=\bm 0 \;\;\text{and}\;\;Q\bm{\psi}_*-\gamma\bm{\psi}_*=\bm 0,
\end{equation}
where $P$ and $Q$ are defined in assumption $(\bf A_1)$.
 Without loss of generality,
we assume that $\bm \varphi_*\ne\bm 0$.
Then it follows from \eqref{limph} that $\gamma$ is an eigenvalue of $P$, which implies that $\gamma=s(P)$ or $\mathcal Re \gamma<s(P)$ from \cite[Corollary 4.3.2]{Smith1995Monotone}.
Since $s(P)=0$ and $\mathcal Re \gamma\ge0$, we have $\gamma=0$. Consequently, $\bm \varphi_*=\kappa_1\bm\xi$ and
 $\bm \psi_*=\kappa_2 \bm \eta$, where $\kappa_1,\kappa_2\in\mathbb C$, and $\kappa_1\ne0$.
Summing the first equation of \eqref{eigen1} over all $j$, we have
\begin{equation*}
  \mu _\la \sum_{j=1}^n\varphi _{\la j}  =\sum_{j=1}^n\la \left[ M_{1j}^{(\la,\beta_\la)}\varphi_{\la j} +M_{2j}^{(\la,\beta_\la)}\psi_{\la j} \right].
\end{equation*}
This, combined with \eqref{esti}, implies that
\begin{equation*}
\limsup_{l\to\infty}\left| \frac{\mu _{\la_l} }{\la_l }\right| \le M^*\ds\f{|\kappa_1|+|\kappa_2|}{|\kappa_1|},
\end{equation*}
which contradicts $\lim_{l\to\infty}|\mu_{\la_l}/\la_l|=\infty$.
Therefore, the claim is true.
\end{proof}

To analyze the stability of
$(\bm x^{(\la,\beta)},\bm y^{(\la,\beta)})$, we need to consider that whether the eigenvalues of \eqref{m4} could pass through the imaginary axis.
It follows from Lemma \ref{l3} that if $\mu={\rm i}\la \nu$ is an eigenvalue of \eqref{m4}, then $\nu$ is bounded for $\la\in(0,\la_1]$, where $\la_1$ is defined in Lemma \ref{l3}.
Substituting $\mu={\rm i}\la \nu\;(\nu\ge0)$ into \eqref{m4}, we have
\begin{equation}\label{v1}
\begin{cases}
\ds{\rm i}\la \nu \varphi _j= \sum_{k=1}^np_{jk}\varphi _k+\la \left[ M_{1j}^{(\la,\beta)} \varphi _j + M_{2j}^{(\la,\beta)} \psi _j \right], &j=1,\dots,n,\\
\ds{\rm i}\la \nu \psi _j = \theta\sum_{k=1}^n q_{jk} \psi _k  + \la \left[ M_{3j}^{(\la,\beta)} \varphi _j -M_{2j}^{(\la,\beta)} \psi _j \right],&j=1,\dots,n. \\
\end{cases}
\end{equation}
Ignoring a scalar factor,
$(\bm \varphi,\bm \psi)^T\in \mathbb{C}^{2n}$ in \eqref{v1} can be represented as
\begin{equation}\label{vps}
\begin{cases}
\bm \varphi ={\left( \varphi_{1}, \cdots,\varphi_{n} \right)}^T=\delta \bm \xi+\bm w, \;\;\text{where}\;\delta\ge0\;\;\text{and}\;\;\bm w\in \mathcal M_{\mathbb C},\\
\bm \psi  ={\left( \psi_{1}, \cdots, \psi_{n} \right)}^T=(s_{1}+{\rm i}s_{ 2})\bm \eta+\bm z,\;\;\text{where}\;s_1,s_2\in\mathbb R\;\;\text{and}\;\;\bm z \in \mathcal M_{\mathbb C},\\
\|\bm \varphi\|_2^{2}+\|\bm \psi\|_2^{2}=\|\bm \xi\|_2^2+\|\bm \eta\|_2^2.
\end{cases}
\end{equation}
Then we see that $(\bm \varphi,\bm \psi,\nu,\beta)$ is a solution of \eqref{v1}, where $\nu\ge0$, $\beta\in\mathcal B$, and $(\bm \varphi,\bm \psi)$ is defined in \eqref{vps}, if and only if
\begin{equation}\label{equi6}
\begin{cases}
 \bm{ H}(\delta,s_1,s_2,\bm w,\bm z,\nu,\beta,\la)=\bm 0\\
\delta\ge0,\; s_1 ,s_2\in\mathbb R,\;\beta\in\mathcal B,\;\nu\ge0,\;\bm w,\bm z\in\mathcal M_{\mathbb C}
 \end{cases}
\end{equation}
is solvable for some value of $(\delta,s_1,s_2,\bm w,\bm z,\nu,\beta)$. Here $\mathcal B$ is defined in Lemma \ref{thglobal}, and
$\bm{ H}(\delta,s_1,s_2,\bm w,\bm z,\nu,\beta,\la):\mathbb R^3\times \left(\mathcal M_{\mathbb C}\right)^2\times \mathbb R\times\mathcal B\times [0,\la_1]\to \left(\mathbb C\times \mathcal M_{\mathbb C}\right)^2\times \mathbb R$ is defined by
\begin{equation*}
 \bm{ H}(\delta,s_1,s_2,\bm w,\bm z,\nu,\beta,\la)=( h_1, h_{21},\dots  h_{2n}, h_3, h_{41},\dots, h_{4n},h_5)^T,
\end{equation*}
where
\begin{equation*}
\begin{split}
 h_1(\delta,s_1,s_2,\bm w,\bm z,\nu,\beta,\la):=& \sum_{j=1}^n  \left[ M_{1j}^{(\la,\beta)}(\delta\xi_j + w_j)
+ M_{2j}^{(\la,\beta)} [ (s_1 + {\rm i} s_2) \eta_j + z_j ]\right]- {\rm i} \nu\delta, \\
 h_{2j}(\delta,s_1,s_2,\bm w,\bm z,h,\beta,\la):=&\sum_{k=1}^n p_{jk}w_k+\la \left[ M_{1j}^{(\la,\beta)}(\delta\xi_j + w_j)
 + M_{2j}^{(\la,\beta)} [ (s_1 + {\rm i} s_2) \eta_j + z_j ]  \right]\\
&- {\rm i} \la\nu(\delta\xi_j + w_j)-\frac{\la }{n} h_1 ,\;\;\;\;j=1,\dots,n,\\
 h_{3}(\delta,s_1,s_2,\bm w,\bm z,\nu,\beta,\la):=& \sum_{j=1}^n  \left[ M_{3j}^{(\la,\beta)}(\delta\xi_j + w_j)
- M_{2j}^{(\la,\beta)} [ (s_1 + {\rm i} s_2) \eta_j + z_j] \right]\\
& -{\rm i}  \nu (s_1 + {\rm i} s_2) , \\
 h_{4j}(\delta,s_1,s_2,\bm w,\bm z,\nu,\beta,\la):=&\theta\sum_{k=1}^n q_{jk}z_k+\la\left[  M_{3j}^{(\la,\beta)}(\delta\xi_j + w_j)
 - M_{2j}^{(\la,\beta)} [(s_1 + {\rm i} s_2) \eta_j + z_j]\right]\\
&-{ \rm i}\la\nu  [ (s_1 + {\rm i} s_2) \eta_j + z_j]-\frac{\la }{n}h_3 ,\;\;\;\;j=1,\dots,n,\\
 h_5(\delta,s_1,s_2,\bm w,\bm z,\nu,\beta,\la):=&\left(\delta^2-1\right)\|\bm \xi\|_2^2+\delta \sum_{j=1}^n\xi_j(w_j+\overline w_j)+\left(s_1^2+s_2^2-1\right)\|\bm \eta\|_2^2\\
&+\|\bm w\|_2^2+\sum_{j=1}^n\eta_j\left[s_1(z_j+\overline z_j)+{\rm i}s_2(\overline z_j- z_j)\right]+\|\bm z \|_2^2,
\end{split}
\end{equation*}
and $M_{1j}^{(\la,\beta)},M_{2j}^{(\la,\beta)},M_{3j}^{(\la,\beta)}$ are defined in \eqref{Mi}.

We first solve \eqref{equi6} for $\la=0$.
\begin{lemma}\label{l5}
Assume that $\la=0$. Then
\eqref{equi6}
 admits a unique solution $$(\delta,s_1,s_2,\bm w,\bm z,\nu,\beta)=(\delta_0,s_{10},s_{20},\bm w_0,\bm z_0,\nu_0,\beta_0),$$ where
\begin{equation*}
\begin{split}
&\nu_0 =\left( \sum_{j=1}^n {a_j} \right)\sqrt {\sum_{j=1}^n  \xi_j^2\eta_j },\;\beta_0 =  \frac{\nu_0^2 + 1}{\sum_{j=1}^n b_j \xi_j },\;\bm w_0=\bm 0,\;\bm z_0=\bm 0,\\
&\delta_0=  \sqrt{\frac{ \|\bm \xi\|_2^2+\|\bm \eta\|_2^2}{\|\bm \xi\|_2^2+\left(1+\frac{1}{\nu^2_0}\right)\|\bm \eta\|_2^2}},\;s_{10}=-\delta_0,\;s_{20}=\frac{\delta_0}{\nu_0}.\\
\end{split}
\end{equation*}
\end{lemma}

\begin{proof}
Substituting $\la =0$ into $ h_{2j}= h_{4j}=0$ for $j=1,\dots,n$, we have $\bm w=\bm w_0=\bm 0$ and $\bm z=\bm z_0=\bm 0$. Then
plugging $\bm w=\bm z=\bm0$ and $\lambda=0$ into $ h_1= h_3=0$, we have

\begin{equation}\label{c1}
\left( {\begin{array}{*{20}{c}}
 \sum_{j=1}^n  M_{1j}^{(0,\beta)}\xi_j -{\rm i}\nu & \sum_{j=1}^n  M_{2j}^{(0,\beta)}  \eta_j \\
 \sum_{j=1}^n  M_{3j}^{(0,\beta)}\xi_j &-\sum_{j=1}^n  M_{2j}^{(0,\beta)}  \eta_j-{\rm i}\nu
\end{array}} \right)\left( {\begin{array}{*{20}{c}}
\delta\\
s_1 + {\rm i} s_2
\end{array}} \right)= \left( {\begin{array}{*{20}{c}}
0\\
0
\end{array}} \right),
\end{equation}
where
\begin{equation}\label{sum-m123}
M_{1j}^{(0,\beta)}=2c_0r_{0\beta}\xi_j\eta_j-\beta b_j-1,\;\;
M_{2j}^{(0,\beta)}=c_0^2\xi_j^2,\;\;M_{3j}^{(0,\beta)}=\beta b_j-2c_0r_{0\beta}\xi_j\eta_j,
\end{equation}
and $c_0$ and $r_{0\beta}$ are defined in \eqref{c0rbeta}. Then we see from \eqref{c0rbeta} that
\begin{equation}\label{summs}
\begin{split}
&\sum_{j=1}^n  M_{1j}^{(0,\beta)}\xi_j=\beta\sum_{j=1}^n b_j\xi_j-1,\;\;\sum_{j=1}^n  M_{2j}^{(0,\beta)} \eta_j=\left(\sum_{j=1}^n a_j\right)^2\sum_{j=1}^n\xi_j^2\eta_j,\\
&\sum_{j=1}^n M_{3j}^{(0,\beta)}\xi_j=-\beta\sum_{j=1}^n b_j\xi_j.
\end{split}
\end{equation}
Then it follows from \eqref {c1} that
\begin{equation}\label{bns12}
\begin{split}
 &\beta=\beta_0 =  \frac{{\left( {\sum_{j=1}^n a_j } \right)}^2 \sum_{j=1}^n \xi_j^2\eta_j  + 1}{\sum_{j=1}^n b_j \xi_j },\;\nu =\nu_0= \left( \sum_{j=1}^n a_j  \right) \sqrt { \sum_{j=1}^n \xi_j ^2 \eta_j },\\
&s_1 =s_{10}= -\delta,\; s_2 =s_{20}= \frac{\delta}{\nu}.
\end{split}
\end{equation}
Again, plugging $\bm w=\bm z=\bm0$ and $\lambda=0$ into $ h_5=0$, we have
\begin{equation*}
\left(\delta^2-1\right)\|\bm \xi\|_2^2+\left(s_{ 1}^2+s_{ 2}^2-1\right)\|\bm \eta\|_2^2=0.
\end{equation*}
This, combined with \eqref{bns12}, implies that
\begin{equation*}
\delta=\delta_0=  \sqrt{\frac{ \|\bm \xi\|_2^2+\|\bm \eta\|_2^2}{\|\bm \xi\|_2^2+\left(1+\frac{1}{\nu^2_0}\right)\|\bm \eta\|_2^2}}.
\end{equation*}
\end{proof}
Now, we solve \eqref{equi6} for $\la>0$.
\begin{theorem}\label{the1}
There exists $\la_2\in(0,\la_1)$, where $\la_1$ is obtained in Lemma \ref{l3}, and a continuously differentiable mapping
$$\la\mapsto(\delta_\la,s_{1\la},s_{2\la},\bm w_{\la},\bm z_\la,\nu_{\la},\beta_\la):
[0,\la_2]\to(0,\infty)\times\mathbb R^2\times \left(\mathcal M_{\mathbb C}\right)^2\times (0,+\infty )\times \mathcal B$$ such that
\eqref{equi6}
 has a unique solution  $(\delta_\la,s_{1\la},s_{2\la},\bm w_{\la},\bm z_\la,\nu_{\la},\beta_\la)$ for $\la\in[0,\la_2]$.
Moreover, for $\la=0$,
\begin{equation*}
\left(\delta_\la,s_{1\la},s_{2\la},\bm w_{\la},\bm z_\la,\nu_{\la},\beta_\la \right)=(\delta_0,s_{10},s_{20},\bm w_0,\bm z_0,\nu_0,\beta_0),
\end{equation*}
where $(\delta_0,s_{10},s_{20},\bm w_0,\bm z_0,\nu_0,\beta_0)$ is defined in Lemma \ref{l5}.
\end{theorem}

\begin{proof}
We first show the existence.
It follows from Lemma \ref{l5} that
$\bm H\left(K_0\right)=\bm 0$,
where $K_0=(\delta_0,s_{10},s_{20},\bm w_0,\bm z_0,\nu_0,\beta_0,0)$.
A direct computation implies that
the Fr\'echet derivative of $\bm{ H}(\delta,s_1,s_2,\bm w,\bm z,\nu,\beta,\la)$ with respect to $(\delta,s_1,s_2,\bm w,\bm z, \nu,\beta)$ at $K_0$  is as follows:
\begin{equation*}
\begin{split}
&\bm K(\hat \delta, \hat s_1,\hat s_2,\bm{ \hat w},\bm{\hat z},\hat \nu,\hat \beta):\mathbb R^3\times\left(\mathcal M_{\mathbb C}\right)^2\times \mathbb R^2\to \left(\mathbb C\times\mathcal M_{\mathbb C}\right)^2\times \mathbb R,\\
&\bm K(\hat \delta, \hat s_1,\hat s_2,\bm{ \hat w},\bm{\hat z},\hat \nu,\hat \beta)=(k_1,k_{21},\dots,k_{2n},k_3,k_{41},\dots,k_{4n},k_5)^T(\hat \delta, \hat s_1,\hat s_2,\bm{ \hat w},\bm{\hat z},\hat \nu,\hat \beta),
\end{split}
\end{equation*}
where
\begin{equation*}
\begin{split}
k_1:=&\sum_{j=1}^n \left[ \left( 2 c_0 r_{0 \beta _0} \xi _j \eta _j - \beta _0 b_j - 1 \right)\left(\hat \delta\xi _j +\hat w_j\right) + {(c_0\xi _j)}^2 ((\hat s_1+{\rm i}\hat s_2 ) \eta _j + \hat z_j) \right]\\
& -{\rm i}\hat \nu \delta_0- {\rm i}\nu_0\hat \delta+\hat \beta \delta_0 \sum_{j=1}^nb_j\xi_j ,\\
k_{2j}:=&\sum_{k= 1}^n {p_{jk}\hat w_k},~~~~~j=1,\dots,n,\\
k_3:=&\sum_{j=1}^n  \left[ (\beta _0  b_j - 2 c_0 r_{0 \beta _0}  \xi _j\eta _j) \left(\hat \delta\xi _j +\hat w_j\right)-  {(c_0\xi _j)}^2 ((\hat s_1+{\rm i}\hat s_2 ) \eta _j+ \hat z_j) \right]\\
&- {\rm i}\nu_0(\hat s_1+ {\rm i}\hat s_2) +\hat \nu\delta_0\left({\rm i}+\frac{1}{\nu_0}\right)-\hat \beta  \delta_0 \sum_{j=1}^nb_j\xi_j , \\
\end{split}
\end{equation*}
\begin{equation*}
\begin{split}
k_{4j}:=&\theta \sum_{k = 1}^n {q_{jk}\hat z_k},~~~~~j=1,\dots,n,\\
k_5:=& 2\delta_0\|\bm \xi\|_2^2 \hat \delta - 2\delta_0\|\bm \eta\|_2^2\hat s_1 + \frac{2\delta_0}{\nu_0}\|\bm \eta\|_2^2\hat s_2+ \delta_0\sum_{j=1}^n\xi_j( \hat w_ j+\overline {\hat w_ j})\\
&+\delta_0\sum_{j=1}^n\eta_j\left[\frac{\rm i}{\nu_0}(\overline{\hat z_ j}- \hat z_ j)- (\hat z_ j+\overline {\hat z_j})\right].
\end{split}
\end{equation*}
If $\bm K( \hat \delta_\la,\hat s_1,\hat s_2,\bm{ \hat w},\bm{\hat z},\hat \nu,\hat \beta)=\bm 0$,  we see that $\bm {\hat w}=\bm 0$ and $\bm {\hat z}=\bm 0$. Plugging $\bm {\hat w}=\bm {\hat z}=\bm 0$ into $k_1=k_3=k_5=0$, we get
\begin{equation*}
D(\hat \delta,\hat s_1,\hat s_2,\hat \nu,\hat \beta)^T=\bm 0,
\end{equation*}
where
\begin{equation*}
D=
\left(\begin{array}{*{20}{c}}
d_{11} &\sum_{j=1}^n  {(c_0\xi _j)}^2 \eta _j &0&0&\delta_0\sum_{j=1}^nb_j\xi_j  \\
-\nu_0&0 &\sum_{j=1}^n  {(c_0\xi _j)}^2 \eta _j &-\delta_0&0\\
d_{31} &-\sum_{j=1}^n  {(c_0\xi _j)}^2 \eta _j &\nu _0&\frac{\delta_0}{\nu _0} &-\delta_0\sum_{j=1}^nb_j\xi_j \\
0&-\nu _0&-\sum_{j=1}^n  {(c_0\xi _j)}^2 \eta _j &\delta_0&0\\
 \|\bm \xi\|_2^2& -\|\bm \eta\|_2^2&\frac{ 1}{ \nu_0}\|\bm \eta\|_2^2&0&0
\end{array} \right), \\
\end{equation*}
and
\begin{equation*}
d_{11}=\sum_{j=1}^n  \left( 2 c_0 r_{0 \beta _0} \xi _j \eta _j - \beta _0 b_j - 1 \right)\xi _j,\;\;
d_{31}=\sum_{j=1}^n  (\beta _0  b_j - 2 c_0 r_{0 \beta _0}  \xi _j\eta _j)\xi _j.
\end{equation*}
A direct computation implies that
\begin{equation*}
|D|=2\delta^2_0\left(\sum_{j=1}^nb_j\xi_j\right)\left(\|\bm\eta\|_2^2+\nu_0^2\|\bm \xi\|_2^2+\nu_0^2\|\bm\eta\|_2^2\right)\ne0,
\end{equation*}
which implies that $\hat \delta=0,\hat  s_1=0,\hat  s_2=0,\hat \nu=0$ and $\hat \beta=0$. Therefore, $\bm K$ is bijection.
It follows from the implicit function theorem that there exists $\la_2\in(0,\la_1)$, where $\la_2$ is sufficiently small,
and a continuously differentiable mapping
$$\la\mapsto(\delta_\la,s_{1\la},s_{2\la},\bm w_{\la},\bm z_\la,\nu_{\la},\beta_\la):
[0,\la_2]\to(0,\infty)\times\mathbb R^2\times \left(\mathcal M_{\mathbb C}\right)^2\times (0,+\infty )\times \mathcal B$$
such that $\bm H\left(\delta_\la,s_{1\la},s_{2\la},\bm w_\la,\bm z_\la,\nu_\la,\beta_\la,\la\right)=\bm 0$, and  for $\la=0$,
\begin{equation*}
\left(\delta_\la,s_{1\la},s_{2\la},\bm w_{\la},\bm z_\la,\nu_{\la},\beta_\la \right)=(\delta_0,s_{10},s_{20},\bm w_0,\bm z_0,\nu_0,\beta_0).
\end{equation*}
Then $\left(\delta_\la,s_{1\la},s_{2\la},\bm w_{\la},\bm z_\la,\nu_{\la},\beta_\la \right)$ is a positive solution of \eqref{equi6} for $\la \in [0,\la_2 ]$.

Now, we show the uniqueness. From the implicit function theorem, we only need to verify that if $\left(\delta^{\la},s_{1}^{\la},s_{2}^{\la},\bm {w}^{\la},\bm{z}^{\la},\nu^{\la},\beta^{\la}\right)$ is a solution of \eqref{equi6},
then
\begin{equation}\label{limv}
\left(\delta^{\la},s_{1}^{\la},s_{2}^{\la},\bm {w}^{\la},\bm{z}^{\la},\nu^{\la},\beta^{\la}\right)\to (\delta_0,s_{10},s_{20},\bm w_0,\bm z_0,\nu_0,\beta_0) \;\;\text{as}\;\;
\la \to 0.
\end{equation}
Since $$\| \delta^{\la}\bm \xi+\bm w^{\la} \|_2^{2}+\|(s^{\la}_1+{\rm i} s^{\la}_2)\bm\eta+\bm z^{\la}\|_2^{2}=\|\bm \xi\|_2^2+\|\bm \eta\|_2^2,$$
we obtain that $\bm {w}^{\la}$, $\bm { z}^{\la},\delta^\la $, $s^{\la}_1$ and $s_2^{\la}$ are bounded. It follows from Lemma \ref{l3} that  $\nu^\la$  is bounded. Then, up to a sequence, we can assume that
\begin{equation*}
\begin{split}
&\lim_{\la\to0} \bm  w^{\la}=\bm w^*,\;\lim_{\la\to0}  \bm {z}^{\la}=\bm z^*,\;\lim_{\la\to0}   s^{\la}_1= s_1^*,\;\lim_{\la\to0}   s^{\la}_2= s_2^*,\\
&\lim_{\la\to0}  \delta^{\la}= \delta^*,\;\lim_{\la\to0}  \beta^{\la}= \beta^*,\;\lim_{\la\to0}  \nu^{\la}= \nu^*.
\end{split}
 \end{equation*}
Taking the limit of \eqref{equi6} as $\la\to0$, we see
that
 $(\delta^*,s^*_1,s^*_2,\bm w^*,\bm z^*,\nu^*,\beta^*)$ is a solution of \eqref{equi6} for $\la=0$. This, combined with
Lemma \ref{l5}, implies that \eqref{limv} holds. This completes the proof.
\end{proof}

Then from Theorem \ref{the1}, we obtain the following result.

\begin{theorem}\label{the1*}
Assume that $\la\in(0,\la_2]$, where $\la_2$ defined in Theorem \ref{the1}. Then $(\bm\varphi,\bm\psi,\nu,\beta)$ solves
\begin{equation*}
\begin{cases}
\left(A_\beta(\la)- {\rm i}\la\nu I\right)(\bm \varphi,\bm \psi)^T=\bm 0,\\
\nu\geq0,\;\beta\in\mathcal B,\;(\bm\varphi,\bm \psi)^T(\ne\bm 0)\in \mathbb C^{2n},\\
\end{cases}
\end{equation*}
 if and only if
$$\nu=\nu_\la,\;\beta=\beta_\la,\;\bm \varphi=\kappa\bm \varphi_\la =\kappa(\delta_\la \bm \xi+\bm w_\la),\;\bm\psi=\kappa\bm\psi_\la=\kappa((s_{1\la}+{\rm i}s_{2\la}) \bm \eta+\bm z_\la),$$
where $A_\beta(\la)$ is defined in \eqref{ph1}, $I$ is the identity matrix, $\kappa$ is a nonzero constant, and $(\nu_{\la},\beta_\la,\delta_\la,s_{1\la},s_{2\la},\bm w_{\la},\bm z_\la)$ is defined in Theorem \ref{the1}.
\end{theorem}

To show that ${\rm i}\la\nu_\la$ is a simple eigenvalue of $A_{\beta_\la}(\la)$, we need to consider the eigenvalues of $A^T_{\beta_\la}(\la)$, where
$A^T_{\beta_\la}(\la)$
is the transpose of $A_{\beta_\la}(\la)$.
Denote
\begin{equation*}
\begin{split}
&\widetilde{\mathcal{M}}_1 =\{(x_1,\dots,x_n)^T \in \mathbb{R}^{n}: \sum_{j=1}^{n}\xi_j x_j =0\},\\
&\widetilde{\mathcal{M}}_2 =\{(x_1,\dots,x_n)^T \in \mathbb{R}^{n}: \sum_{j=1}^{n} \eta_j x_j =0\}.
\end{split}
\end{equation*}
Then
\begin{equation*}
\mathbb R^n={\rm span}\left\{\left( 1, \cdots,1 \right)^T\right\}\oplus  \widetilde{\mathcal{M}}_1={\rm span}\left\{\left( 1, \cdots,1 \right)^T\right\}\oplus  \widetilde{\mathcal{M}}_2.
\end{equation*}
\begin{lemma}\label{ad}
Let $ A_{\beta_\la}^T(\la)$ be the transpose of $A_{\beta_\la}(\la)$, where $\la\in(0,\la_2]$ and $\la_2$ defined in Theorem \ref{the1}. Then $-{\rm i}\la\nu_\la$ is an eigenvalue of $A_{\beta_\la}^T(\la)$, and
$$\mathcal N [ A_{\beta_\la}^T(\la) +  {\rm i}\la \nu_\la I] ={\rm span}  [(\bm {\tilde\varphi}_\la,\bm{\tilde \psi}_\la)^T].$$
Moreover, ignoring a scalar factor, $(\bm {\tilde\varphi}_\la,\bm{\tilde \psi}_\la)$ can be represented as
\begin{equation}\label{trans-vps}
\begin{cases}
\bm  {\tilde\varphi}_\la =\tilde\delta_\la {\left( 1, \cdots,1 \right)}^T+\bm {\tilde w}_\la, \;\;\text{where}\;\;\tilde \delta_\la\ge0\;\;\text{and}\;\;\bm {\tilde w}_\la\in\left(\widetilde{ \mathcal M}_{1}\right)_\mathbb C,\\
\bm {\tilde \psi}_\la  =(\tilde s_{1\la}+{\rm i}\tilde s_{ 2\la}) {\left( 1, \cdots,1 \right)}^T+\bm{ \tilde z}_\la,\;\;\text{where}\;\; \tilde s_{1\la},\tilde s_{2\la}\in\mathbb R\;\;\text{and}\;\;\bm {\tilde z}_\la \in\left(\widetilde{ \mathcal M}_{2}\right)_\mathbb C,\\
\|\bm  {\tilde\varphi}_\la\|_2^{2}+\|\bm {\tilde \psi}_\la\|_2^{2}=2n,
\end{cases}
\end{equation}
and
$
\left(\tilde \delta_\la,\tilde s_{1\la},\tilde s_{2\la},\bm {\tilde w}_{\la},\bm {\tilde z}_\la \right)=(\tilde\delta_0, \tilde s_{10},\tilde s_{20},\bm{ \tilde w}_0,\bm {\tilde z}_0 )$ for $\la=0$,
where
\begin{equation*}
\tilde s_{10}=\frac{\tilde \delta_0 \nu_0^2  }{\nu_0^2 + 1},
\;\tilde s_{20}=\frac{ \tilde \delta_0 \nu_0}{\nu_0^2 + 1},\;
\bm{ \tilde w}_0=\bm 0,\;\bm {\tilde z}_0=\bm 0,\;
\tilde\delta_0=  \sqrt{\frac{2\nu_0^2+2}{2\nu_0^2+1}},
\end{equation*}
and $\nu_0$ is defined in Lemma \ref{l5}.
\end{lemma}

\begin{proof}
It follows from Theorem \ref{the1*} that, for $\la\in(0,\la_2]$, $\pm {\rm i}\la\nu_\la$ is an eigenvalue of $A_{\beta_\la}(\la)$ and $\mathcal N [ A_{\beta_\la}(\la) -{\rm i}\la \nu_\la I]$ is one-dimensional. Then
$-{\rm i}\la\nu_\la$ is an eigenvalue of $A^T_{\beta_\la}(\la)$, $\mathcal N [ A_{\beta_\la}^T(\la) +  {\rm i}\la \nu_\la I] ={\rm span}  [(\bm {\tilde\varphi}_\la,\bm{\tilde \psi}_\la)^T]$, and $(\bm {\tilde\varphi}_\la,\bm{\tilde \psi}_\la)$ can be represented as \eqref{trans-vps}.
It follows from \eqref{trans-vps} that
$\tilde\delta_\la$, $\tilde s_{1\la}$, $\tilde s_{2\la}$, $\bm {\tilde w}_\la$ and $\bm {\tilde z}_\la$ are bounded.
Then, up to a sequence, we assume that
\begin{equation*}
\begin{split}
&\lim_{\la\to0}\bm{\tilde\varphi}_{\la}=\bm{\tilde{\varphi}}_*,\;\lim_{\la\to0}\bm{\tilde\psi}_{\la}=\bm{\tilde{\psi}}_*,\;
\lim_{\la\to0} \bm {\tilde w}_\la=\bm{ \tilde w}_0,\;\lim_{\la\to0} \bm {\tilde z}_\la=\bm{ \tilde z}_0,\\
&\lim_{\la\to0}   \tilde s_{1\la}=\tilde s_{10},\;\lim_{\la\to0}   \tilde s_{2\la}=\tilde s_{20},\;
\lim_{\la\to0}  \tilde\delta_\la=\tilde \delta_0,
\end{split}
 \end{equation*}
where
\begin{equation}\label{starp}
\|\bm{\tilde{\varphi}}_*\|_2^2+\|\bm{\tilde{\psi}}_*\|_2^2=2n.
\end{equation}
Note that $\lim_{\la\to0}  \beta_{\la}=\beta_0$ and $\lim_{\la\to0}  \nu_{\la}= \nu_0$, where
 $\nu_0,\beta_0$ are defined in Lemma \ref{l5}.
Taking the limit of
\begin{equation}\label{det}
\left( A_{\beta_{\la}}^T({\la})+ {\rm i}{\la} \nu_{\la} I \right) (\bm {\tilde\varphi}_{\la}, \bm{\tilde \psi}_{\la})^T=\bm 0,
\end{equation}
as $\la\to0$, we have $P^T\bm{\varphi}_*=\bm 0$ and $Q^T\bm{\psi}_*=\bm 0$, which yield $\bm{\tilde w}_{0}=\bm 0$ and $\bm{\tilde z}_{0}=\bm 0$.
Noticing that $P\bm{\xi}=Q\bm{\eta}=\bm 0$, we see from \eqref{det} that
\begin{equation*}
\left( {\begin{array}{*{20}{c}}
 \sum_{j=1}^n  M_{1j}^{(0,\beta_0)}\xi_j +{\rm i}\nu_0 & \sum_{j=1}^n  M_{3j}^{(0,\beta_0)}\xi_j\\
 \sum_{j=1}^n  M_{2j}^{(0,\beta_0)}  \eta_j &-\sum_{j=1}^n  M_{2j}^{(0,\beta_0)}  \eta_j+{\rm i}\nu_0
\end{array}} \right)\left( {\begin{array}{*{20}{c}}
\tilde\delta_0\\
\tilde s_{10} + {\rm i} \tilde s_{20}
\end{array}} \right)= \left( {\begin{array}{*{20}{c}}
0\\
0
\end{array}} \right).
\end{equation*}
This, combined with Eqs. \eqref{summs} and \eqref{bns12}, implies that
\begin{equation}\label{ti}
\tilde s_{10}=\frac{\tilde \delta_0 \nu_0^2  }{\nu_0^2 + 1},
\;\tilde s_{20}=\frac{ \tilde \delta_0 \nu_0}{\nu_0^2 + 1}.
\end{equation}
Note from \eqref{starp} that $\tilde\delta_0^2+\tilde s_{ 10}^2+\tilde s_{ 20}^2=2.$
This, combined with \eqref{ti}, implies that
$$\tilde\delta_0=  \sqrt{\frac{2\nu_0^2+2}{2\nu_0^2+1}}.$$
This completes the proof.
\end{proof}

Now, we show that ${\rm i} \la \nu_\la$ is simple.

\begin{theorem}\label{sim}
Assume that $\la \in (0, \la_2]$, where $\la_2$ is sufficiently small. Then ${\rm i} \la \nu_\la$ is a simple eigenvalue of $ A_{\beta_\la}(\la) $, where $ A_{\beta}(\la) $ is defined in \eqref{ph1}.
\end{theorem}

\begin{proof}
It follows from Theorem \ref{the1*} that $\mathcal N [ A_{\beta_\la}(\la) -  {\rm i}\la \nu_\la I] ={\rm span}  [(\bm \varphi_\la, \bm \psi_\la)^T]$, where $\bm \varphi_\la, \bm \psi_\la$ is defined in Theorem \ref{the1}. Then we show that
\begin{equation*}
\mathcal N [ A_{\beta_\la}(\la) -  {\rm i}\la \nu_\la I]^2 =\mathcal N [ A_{\beta_\la}(\la) -  {\rm i}\la \nu_\la I].
\end{equation*}
If $\bm \Psi =( \bm \Psi_1,  \bm \Psi_2)^T \in \mathcal N [ A_{\beta_\la}(\la) -  {\rm i}\la \nu_\la I]^2$, where $\bm \Psi_1=(\Psi_{11},\dots,\Psi_{1n})^T\in{\mathbb C}^n$ and $\bm \Psi_2=(\Psi_{21},\dots,\Psi_{2n})^T\in{\mathbb C}^n$, then
\begin{equation*}
 [ A_{\beta_\la}(\la) -  {\rm i}\la \nu_\la I] \bm \Psi  \in \mathcal N [ A_{\beta_\la}(\la) -  {\rm i}\la \nu_\la I]={\rm span}  [(\bm \varphi_\la, \bm \psi_\la)^T],
\end{equation*}
which implies that there exists a constant $s$ such that
\begin{equation*}
 [ A_{\beta_\la}(\la) -  {\rm i}\la \nu_\la I] \bm \Psi  =s (\bm \varphi_\la, \bm \psi_\la)^T.
\end{equation*}
That is,
\begin{equation}\label{Psi2}
\begin{cases}
\ds  s \varphi_{\la j} = \sum_{j=1}^np_{ij} \Psi_{1j} +\la \left[ M_{1j}^{(\la,{\beta_\la})} \Psi_{1j} +M_{2j}^{(\la,{\beta_\la})}  \Psi_{2j} \right]- {\rm i}\la \nu_\la \Psi_{1j}, \;j=1,\dots,n,\\
\ds s \psi_{\la j}= \theta\sum_{j=1}^n q_{ij} \tilde \Psi_{2j}  + \la \left[ M_{3j}^{(\la,{\beta_\la})} \Psi_{1j} -M_{2j} ^{(\la,{\beta_\la})} \Psi_{2j}\right]- {\rm i}\la \nu_\la \Psi_{2j},\;j=1,\dots,n.
\end{cases}
\end{equation}
It follows from Lemma \ref{ad} that
$$\left( A_{\beta_\la}^T(\la)+ {\rm i}\la \nu_\la I \right) (\bm {\tilde\varphi}_\la, \bm{\tilde \psi}_\la)^T=\bm 0.$$
Then, multiplying the first and second equation of \eqref{Psi2} by $ \overline{\tilde\varphi _{\la j}}$ and $ \overline{\tilde\psi_{\la j}}$, respectively,
and summing the results over all $j$, we have
\begin{equation*}
\begin{split}
0&=\left\langle \left( A_{\beta_\la}^T(\la)+{\rm i}\la \nu_\la I \right)(\bm {\tilde\varphi}_\la, \bm{\tilde \psi}_\la)^T ,\bm \Psi  \right\rangle=\left\langle (\bm {\tilde\varphi}_\la,\bm{\tilde \psi}_\la)^T, \left( A_{\beta_\la}(\la)- {\rm i}\la \nu_\la I \right)\bm \Psi    \right\rangle \\
&=s \sum_{j=1}^n \left( \overline{\tilde\varphi_{\la j}} \varphi_{\la j}+  \overline{ \tilde\psi_{\la j}} \psi_{\la j}\right).\\
\end{split}
\end{equation*}
It follows from Theorem \ref{the1} and Lemma \ref{ad} that
\begin{equation}\label{ptilde}
\begin{split}
\lim_{\la\to 0} \sum_{j=1}^n \left( \overline{\tilde\varphi_{\la j}} \varphi_{\la j}+\overline{ \tilde\psi_{\la j} }\psi_{\la j}\right)&=\delta_0 \tilde\delta_0 \sum_{j=1}^n\left(\xi_j+ \frac{\nu_0}{\nu_0^2+1}(\nu_0-{\rm i})(\frac{{\rm i}}{\nu_0}-1)\eta_j  \right)\\
&=2\delta_0 \tilde\delta_0  \frac{{\rm i}\nu_0+1}{\nu_0^2+1}\neq0,
\end{split}
\end{equation}
which implies that $s=0$ for $\la\in(0,\la_2]$, where $\la_2$ is sufficiently small. Therefore,
\begin{equation*}
\mathcal N [ A_{\beta_\la}(\la) -  {\rm i}\la \nu_\la I]^j =\mathcal N [ A_{\beta_\la}(\la) -  {\rm i}\la \nu_\la I],\;\;j=2,3,\cdots.
\end{equation*}
This completes the proof.
\end{proof}

It follows from Theorem \ref{sim} that $\mu={\rm i}\la\nu_\la$ is a simple eigenvalue of $ A_{\beta_\la}(\la) $ for fixed $\la\in(0,\la_2]$. Then, by using the implicit function
theorem, we see that there exists a neighborhood $V_1\times V_2\times O$ of $ (\bm \varphi_\la,\bm \psi_\la, {\rm i}\nu_\la,\beta_\la)$ ( $V_1$  is defined as the neighborhood of $(\bm \varphi_\la,\bm \psi_\la)$ ) and a continuously differentiable function  $ (\bm \varphi(\beta),\bm \psi(\beta), \mu(\beta)):O\to V_1\times V_2$ such that $\mu(\beta_\la)={\rm i}\nu_\la,\;(\bm \varphi(\beta_\la),\bm \psi(\beta_\la))=(\bm \varphi_\la,\bm \psi_\la)$. Moreover, for each $\beta \in O$, the only eigenvalue of $A_{\beta}(\la)$ in $V_2$ is $\mu(\beta)$, and
\begin{equation}\label{tra}
\left(A_{\beta}(\la)- \mu(\beta) I\right)(\bm \varphi(\beta),\bm \psi(\beta))^T=\bm 0.
\end{equation}
Then, we show that the following transversality condition holds.

\begin{theorem}\label{the2}
For $\la \in (0,\la_2 ]$, where $\la_2$ is sufficiently small,
\begin{equation*}
 \left.\frac{ d \mathcal Re [\mu(\beta)]}{ d \beta}\right|_{\beta=\beta_\la} > 0.
\end{equation*}
\end{theorem}

\begin{proof}
Differentiating \eqref{tra} with respect to $\beta$ at $\beta=\beta_\la,$ we have
\begin{equation}\label{tra1}
\left.\frac{ d \mu}{ d \beta}\right|_{\beta=\beta_\la} (\bm \varphi_\la,\bm \psi_\la)^T= \left(A_{\beta_\la}(\la)- {\rm i}\nu_\la I\right)\left.\left(\frac{d\bm \varphi }{ d \beta}, \frac{d\bm \psi }{ d \beta}  \right)^T\right|_{\beta=\beta_\la}+\left.\frac{d A_{\beta}(\la) }{ d \beta} \right|_{\beta=\beta_\la} (\bm \varphi_\la,\bm \psi_\la)^T.
\end{equation}
Note that
\begin{equation*}
\begin{split}
&\left\langle(\bm {\tilde\varphi}_\la, \bm{\tilde \psi}_\la)^T ,\left(A_{\beta_\la}(\la)- {\rm i}\nu_\la I\right)\left.\left(\frac{d\bm \varphi }{ d \beta}, \frac{d\bm \psi }{ d \beta}  \right)^T\right|_{\beta=\beta_\la}\right\rangle\\
&=\left\langle\left( A_{\beta_\la}^T(\la)+{\rm i}\la \nu_\la I \right)({\bm {\tilde\varphi}_\la}, {\bm{\tilde \psi}_\la})^T, \left.\left(\frac{d\bm \varphi }{ d \beta}, \frac{d\bm \psi }{ d \beta}  \right)^T\right|_{\beta=\beta_\la} \right\rangle=0,
\end{split}
\end{equation*}
where $\bm {\tilde\varphi}_\la$ and $\bm{\tilde \psi}_\la$ are defined in Lemma \ref{ad}.
Then we see from \eqref{tra1} that
\begin{equation}\label{ptilde2}
\left.\frac{ d \mu}{ d \beta}\right|_{\beta=\beta_\la} \left\langle({\bm {\tilde\varphi}_\la}, {\bm{\tilde \psi}_\la})^T ,(\bm {\varphi}_\la, \bm{ \psi}_\la)^T \right\rangle=\left\langle({\bm {\tilde\varphi}_\la}, {\bm{\tilde \psi}_\la})^T ,\left.\frac{d A_{\beta}(\la) }{ d \beta} \right|_{\beta=\beta_\la}(\bm {\varphi}_\la, \bm{ \psi}_\la)^T \right\rangle.\\
\end{equation}
It follows from Theorem \ref{thglobal} that
$(\bm x^{(\la,\beta)},\bm y^{(\la,\beta)})$ is continuously differentiable on $[0,\lambda_2]\times \mathcal B$. This, combined the fact that
$\lim_{\la\to0}\beta_\la=\beta_0$, implies that
\begin{equation}\label{limm123}
\lim_{\la\to0}\left.\ds\f{dM^{(\la,\beta)}_{lj}}{d\beta}\right|_{\beta=\beta_\la}=\left.\ds\f{dM^{(0,\beta)}_{lj}}{d\beta}\right|_{\beta=\beta_0}
\;\;\text{for}\;\;l=1,2,3\;\;\text{and}\;\;j=1,\dots,n,
\end{equation}
where $M^{(\la,\beta)}_{lj}$ is defined in \eqref{Mi}.
Then we see from \eqref{c0rbeta} and \eqref{limm123} that
$$
\lim_{\la\to0}\left.\frac{d A_{\beta}(\la) }{ d \beta} \right|_{\beta=\beta_\la}=\la\left( {\begin{array}{cc}
S&O_{n \times n}\\
-S&O_{n \times n}
\end{array}} \right),
$$
where
$$S={\rm diag}\left( \frac{2\sum_{j=1}^nb_j\xi_j}{\sum_{j=1}^n \xi_j^2\eta_j}\xi_j\eta_j-b_j \right),$$
and $O_{n \times n}$ is a zero matrix of $n\times n.$
It follows from Theorem \ref{the1} and Lemma \ref{ad} that
\begin{equation*}
\lim_{\la\to0}\ds\frac{1}{\la}\left\langle({\bm {\tilde\varphi}_\la}, {\bm{\tilde \psi}_\la})^T ,\left.\frac{d A_{\beta}(\la) }{ d \beta} \right|_{\beta=\beta_\la}(\bm {\varphi}_\la, \bm{ \psi}_\la)^T \right\rangle
=\ds \frac{\delta_0 \tilde\delta_0\sum_{j=1}^n  b_j \xi _j}{\nu_0^2+1}(1+{\rm i}\nu_0).
\end{equation*}
This, combined with \eqref{ptilde} and \eqref{ptilde2}, yields
$$ \lim_{\la\to 0} \frac{1}{\la} \left.\frac{ d \mathcal Re [\mu(\beta)]}{ d \beta}\right|_{\beta=\beta_\la}=\frac{1}{2} \sum_{j=1}^n  b_j \xi _j > 0.$$
\end{proof}

\begin{figure}[htbp]
\centering
\setlength{\abovecaptionskip}{-1cm}
\includegraphics[width=0.5\textwidth]{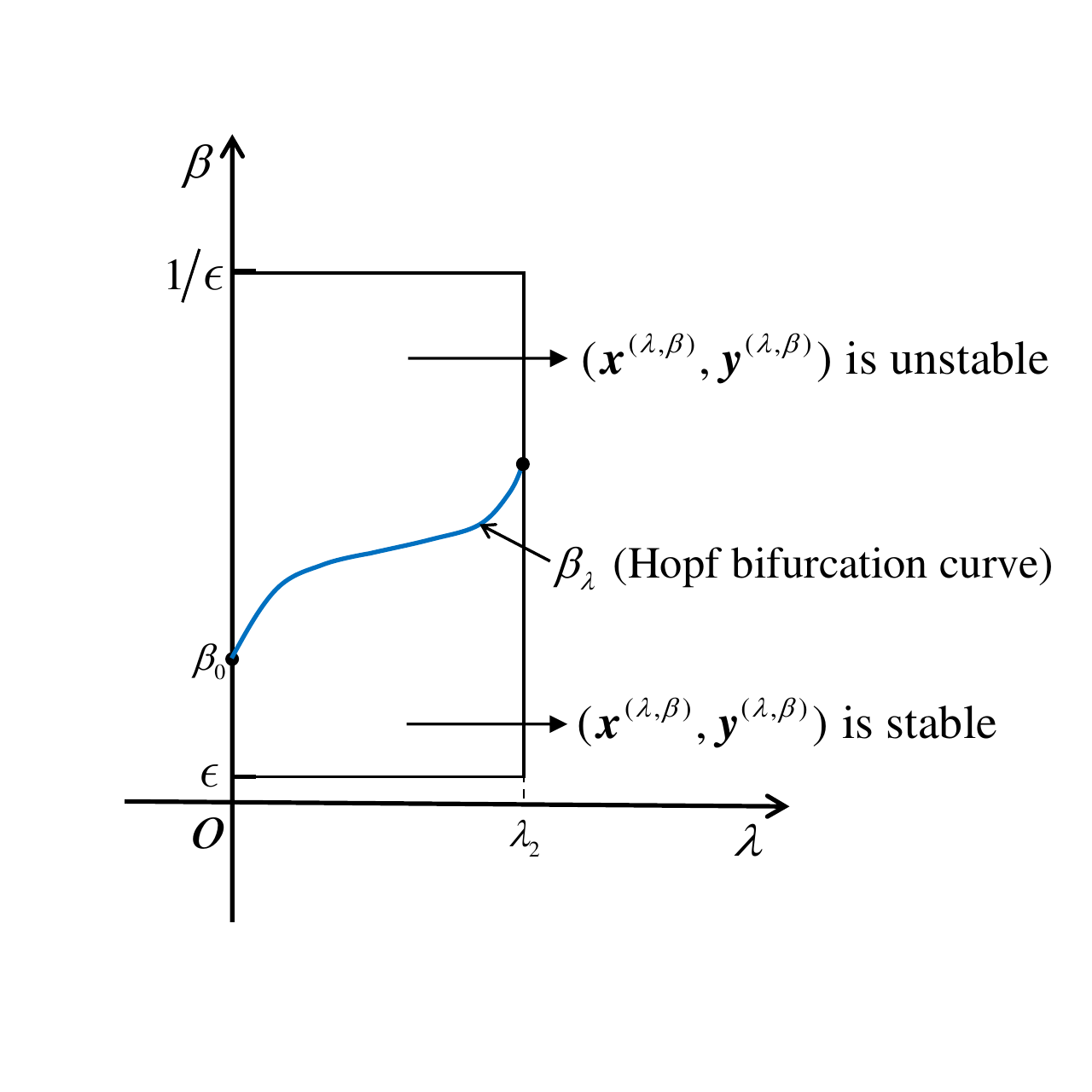}
\caption{Bifurcation diagram of model \eqref{m2} with respect to $\beta$ when $ \la$ is small. \label{fig2}}
\end{figure}
From Theorems \ref{thglobal}, \ref{the1},  \ref{sim} and \ref{the2},  we can obtain the following results on the dynamics of model \eqref{m2}, see also Fig. \ref{fig2}.
\begin{theorem} \label{ed}
Assume that $0<\epsilon\ll1$. Then there exists $\la_2>0$ (depending on $\epsilon$) and a continuously differentiable mapping
$$\la\mapsto\beta_\la:
[0,\la_2]\to\mathcal B=[\epsilon,1/\epsilon]$$ such that, for each $\la\in(0,\la_2]$, the positive equilibrium
$(\bm x^{(\la,\beta)},\bm y^{(\la,\beta)})$ of model \eqref{m2} is locally asymptotically stable when $\beta \in [\epsilon,\beta_\la)$ and unstable when $\beta \in (\beta_\la, \frac{1}{\epsilon}]$. Moreover, system \eqref{m2} undergoes a Hopf bifurcation at $(\bm x^{(\la,\beta)},\bm y^{(\la,\beta)})$ when $\beta=\beta_\la$.
\end{theorem}
\begin{proof}
Note from \cite[Chapter 2, Theorem 5.1]{Kato1995} that the eigenvalues of $A_{\beta}(\la)$ are continuous with respect to $\beta.$ We only need to show that there exists $\la_2>0$, which depends on $\epsilon$, such that
\begin{equation*}
\sigma( A_\beta(\la))\subset \{x+{\rm i}y : x,y\in\mathbb{R},x < 0  \} \; \text{for}\;  \la\in(0,\la_2] \; \text{and}\;  \beta=\epsilon.
\end{equation*}
If it is not true, then there exists a sequence $\{\la_l\}_{l=1}^\infty$  such that $\lim_{l\to\infty}\la_l=0$, and
\begin{equation}\label{Aeps}
\left( A_{\epsilon}({\la_l})-  \mu I \right) (\bm {\varphi}, \bm{ \psi})^T=\bm 0
\end{equation}
is solvable for some value of $(\mu_{\la_l},\bm {\varphi}_{\la_l},\bm{ \psi}_{\la_l})$ with $\mathcal Re \mu_{\la_l }\geq 0$ and $ (\bm\varphi_{\la_l},\bm \psi_{\la_l})^T(\ne\bm 0)\in \mathbb C^{2n}$.
Here $A_{\epsilon}(\la_l)=\left.A_{\beta}(\la_l)\right|_{\beta=\epsilon}$.
Substituting $(\mu,\bm {\varphi},\bm{ \psi})=(\mu_{\la_l},\bm {\varphi}_{\la_l},\bm{ \psi}_{\la_l})$ into \eqref{Aeps}, we have
\begin{equation}\label{Aeps1}
\begin{cases}
\ds\mu_{\la_l} \varphi _{{\la_l} j}= \sum_{k=1}^np_{jk}\varphi_{{\la_l} k} +\la_l \left[ M_{1j}^{({\la_l},\epsilon)} \varphi _{{\la_l} j} + M_{2j}^{({\la_l},\epsilon)}\psi _{{\la_l} j} \right], \;\;j=1,\dots,n, \\
\ds\mu_{\la_l} \psi _{{\la_l} j}= \theta\sum_{k=1}^n q_{jk} \psi_{{\la_l} k}  + \la_l \left[ M_{3j}^{({\la_l},\epsilon)}\varphi_{{\la_l} j} -M_{2j}^{({\la_l},\epsilon)}\psi_{{\la_l} j} \right],\;\;j=1,\dots,n.\\
\end{cases}
\end{equation}
Note from Lemma \ref{l3} that $\left|\mu_{\la_ l}/{\la_l}\right| $ is bounded. Then, up to a subsequence, we assume that $\lim_{l\to\infty} \mu_{\la_l }/{\la_l }=\mu^*$ with $\mathcal Re \mu^*\geq 0$. As in the proof of Lemma \ref{ad}, we see that $(\bm \varphi_{\la l },\bm \psi_{\la l})^T(\neq\bm 0)\in \mathbb{C}^{2n}$ can be represented as
\begin{equation}\label{Aeps2}
\begin{cases}
\bm \varphi_{\la_l} =\delta_{\la_l} \bm \xi+\bm w_{\la_l}, \;\;\text{where}\;\;\delta_{\la_l}\ge0\;\;\text{and}\;\;\bm w_{\la_l}\in \mathcal M_{\mathbb C},\\
\bm \psi_{\la_l}  =({s_1}_{\la_l}+{\rm i}{s_2}_{\la_l})\bm \eta+\bm z_{\la_l},\;\;\text{where}\;\;{s_1}_{\la_l},{s_2}_{\la_l}\in\mathbb R\;\;\text{and}\;\;\bm z_{\la_l} \in \mathcal M_{\mathbb C},\\
\|\bm \varphi_{\la_l}\|_2^{2}+\|\bm \psi_{\la_l}\|_2^{2}=\|\bm \xi\|_2^2+\|\bm \eta\|_2^2,
\end{cases}
\end{equation}
and $\bm w_{\la_l} $ and $ \bm z_{\la_l}$ satisfy
$$\lim_{l\to\infty}\bm w_{\la_l} =\bm 0,\;\;\lim_{l\to\infty}\bm z_{\la_l} =\bm 0.$$
Up to a subsequence, we also assume that $\lim_{l\to\infty} \delta_{\la_l } =\delta^*$, $\lim_{l\to\infty} {s_1}_{\la_l } =s_1^*$ and $\lim_{l\to\infty}{s_2}_{\la_l }=s_2^*$. Dividing the first and second equation of \eqref{Aeps1} by $\la_l$, respectively, summing the results over all $j$, and taking the limits as $l\to\infty$,  we have
\begin{equation*}
\begin{cases}
\ds\mu^* \delta^*= \delta^* \sum_{j=1}^n M_{1j}^{(0,\epsilon)} \xi_j +  \left(s_1^*+{\rm i}s_2^*   \right)\sum_{j=1}^n M_{2j}^{(0,\epsilon)}\eta_j, \\
\ds\mu^* \left(s_1^*+{\rm i}s_2^*   \right)=   \delta^* \sum_{j=1}^n M_{3j}^{(0,\epsilon)}\xi_j  - \left(s_1^*+{\rm i}s_2^*   \right) \sum_{j=1}^n M_{2j}^{(0,\epsilon)}\eta_j .\\
\end{cases}
\end{equation*}
It follows from \eqref{Aeps2} that at least one of
$\delta^*$ and $s_1^*+{\rm i}s_2^* $ is not equal to zero. Consequently, $\mu^*$ is an eigenvalue of
matrix
$$\left( {\begin{array}{cc}
 \sum_{j=1}^n  M_{1j}^{(0,\epsilon)}\xi_j & \sum_{j=1}^n  M_{2j}^{(0,\epsilon)}  \eta_j \\
 \sum_{j=1}^n  M_{3j}^{(0,\epsilon)}\xi_j &-\sum_{j=1}^n  M_{2j}^{(0,\epsilon)}  \eta_j
\end{array}} \right).$$
It follows from \eqref{sum-m123} and \eqref{summs} that, for sufficiently small $\epsilon$, $$\sum_{j=1}^n M_{1j}^{(0,\epsilon)}\xi_j-\sum_{j=1}^n  M_{2j}^{(0,\epsilon)}  \eta_j<0,$$ which contradicts $\mathcal Re \mu^*\geq 0$. This completes the proof.
\end{proof}
Noticing that model \eqref{m2} is equivalent to the original model \eqref{m1}, we have the following result.
\begin{theorem}
Assume that $0<\epsilon\ll1$, and $(\bf A_1)-(\bf A_2)$ hold. Then there exists $d_*>0$ (depending on $\epsilon$) and a continuously differentiable mapping
$$d_1\mapsto\beta^{d_1}:
[d_*,\infty)\to\mathcal B=[\epsilon,1/\epsilon]$$ such that, for each $d_1\in[d_*,\infty)$, the unique positive equilibrium
of model \eqref{m1} is locally asymptotically stable when $\beta \in [\epsilon,\beta^{d_1})$ and unstable when $\beta \in (\beta^{d_1}, \frac{1}{\epsilon}]$. Moreover, system \eqref{m1} undergoes a Hopf bifurcation when $\beta=\beta^{d_1}$.
\end{theorem}
\section{The effect of the coupling matrices}
In this section, we show the effect of the coupling matrices on the Hopf bifurcation value. Moreover, some numerical simulations are given to illustrate the theoretical results.
For simplicity, we consider a special case, where $P=Q$ and the boxes are all identical (that is,
$a_i=a$ and $b_i=1$ for $i=1,\dots,n$). Then, model \eqref{m2} is reduced to the following form:
\begin{equation}\label{e1}
\begin{cases}
\ds\frac{d x_j}{d t}= \sum_{k=1}^np_{jk}x_k+\la \left[ a - (\beta + 1)x_j + x_j^2y_j\right], &j=1,\dots,n,\;\;t>0,\\
\ds\frac{d y_j}{d t} = \theta\sum_{k=1}^n p_{jk}y_k  + \la \left(\beta x_j - x_j^2y_j\right),&j=1,\dots,n,\;\;t>0, \\
\bm x(0)=\bm x_0\ge(\not\equiv)\bm0,\;\bm y(0)=\bm y_0\ge(\not\equiv)\bm0.
\end{cases}
\end{equation}

From Theorem \ref{ed}, we have the following result on the dynamics of \eqref{e1}.

\begin{proposition} \label{pro2}
Assume that $0<\epsilon\ll1$. Then there exists $\la_2>0$ (depending on $\epsilon$) and a  Hopf bifurcation curve:
$$\la\mapsto\beta_\la:
[0,\la_2]\to\mathcal B=[\epsilon,1/\epsilon]$$ such that, for each $\la\in(0,\la_2]$, the positive equilibrium
$(\bm x^{(\la,\beta)},\bm y^{(\la,\beta)})$ of system \eqref{e1} is locally asymptotically stable when $\beta \in [\epsilon,\beta_\la)$ and unstable when $\beta \in (\beta_\la, \frac{1}{\epsilon}]$. Moreover, system \eqref{e1} undergoes a Hopf bifurcation at $(\bm x^{(\la,\beta)},\bm y^{(\la,\beta)})$ when $\beta=\beta_\la$, and $\beta_\la$ satisfies
\begin{equation}\label{examHopf}
\lim_{\la\to0} \beta_\la=\beta_0=1+ (n a)^2\left(\sum_{j=1}^n  \xi_j^3\right),
\end{equation}
where $(\xi_1,\dots,\xi_n)$ satisfying \eqref{xi} is the
corresponding eigenfunction of $P$ with respect to eigenvalue $0$.
\end{proposition}

 Now, we show the effect of the coupling matrix on the Hopf bifurcation value.
 We recall the $P=(p_{jk})$ is line-sum symmetric matrix if $\sum_{k\ne j} p_{jk}=\sum_{k\ne j}p_{kj}$ for all $j=1,\dots,n$.
\begin{proposition} \label{pro3}
Denote $\beta_0$ by $\beta_0^{L}$ (respectively, $\beta_0^{NL}$) when $P=(p_{jk})$ is line-sum symmetric (respectively, not line-sum symmetric).
Then $\beta_0^L<\beta_0^{NL}$.
\end{proposition}
\begin{proof}
It follows from \eqref{examHopf} that $\beta_0=1+ (n a)^2\left(\sum_{j=1}^n  \xi_j^3\right)$.
A direct computation implies that $$\bm \xi=(\xi_1,\dots,\xi_n)^T=\left(\frac{1}{n},\dots, \frac{1}{n}\right)^T,$$ if and only if
$P$ is line-sum symmetric. Then we see that $\beta_0^L=1+a^2$.
It follows from \eqref{xi} and H\"{o}lder inequality that
\begin{equation}\label{Holder}
1=\left(\sum_{j=1}^n  \xi_j\right)^3\leq\left(\sum_{j=1}^n 1^\frac{3}{2}\right)^2\left(\sum_{j=1}^n \xi_j^3\right)= n^2\left(\sum_{j=1}^n \xi_j^3\right),
\end{equation}
and the equality holds if and only if  $\xi_j=1/n$ for $j=1\cdots n$.
Then, we see from \eqref{examHopf} and \eqref{Holder} that
$\beta_0^L<\beta_0^{NL}$.
\end{proof}
From Propositions \ref{pro2} and \ref{pro3}, we see that the Hopf bifurcation value with line-sum symmetric coupling matrix is smaller than that for the non-line-sum symmetric case, see Fig. \ref{fig3}.
\begin{figure}[htbp]
\centering
\setlength{\abovecaptionskip}{-0.5cm}
\includegraphics[width=0.5\textwidth]{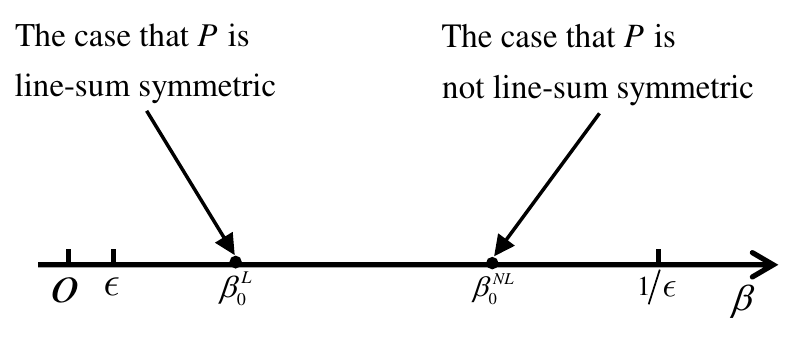}
\caption{The effect of coupling matrices on Hopf bifurcation value.\label{fig3}}
\end{figure}

This phenomenon can also be illustrated numerically. We consider model \eqref{e1}, and choose the parameters and initial values as follows:
\begin{equation*}
n=3,\;\;\beta=2.05,\;\; \la=0.1,\;\;\theta=1,\;\;a=1,\;\;x_j(0)=y_j(0)=1,(j=1,2,3).
\end{equation*}
Here we choose two different coupling matrices:
$$P^L:=\left( {\begin{array}{ccc}
   -2& 1& 1\\
    1&-3& 2\\
    1& 2&-3
\end{array}} \right),\;\;P^{NL}:=\left( {\begin{array}{ccc}
   -3& 2& 3\\
    2&-3& 2\\
    1& 1&-5
\end{array}} \right),$$
where $P^L$ is line-sum symmetric, and $P^{NL}$ is not
line-sum symmetric.
Then when $\beta(=2.05)$ is fixed, model \eqref{e1} admits a positive periodic solution for $P=P^L$, while the positive equilibrium is stable for $P=P^{NL}$, see Fig. \ref{fig6}. This shows that the Hopf bifurcation value of model \eqref{e1} for $P=P^L$ is smaller than that for $P=P^{NL}$.
\begin{figure}[htbp]
\centering
\includegraphics[width=0.4\textwidth]{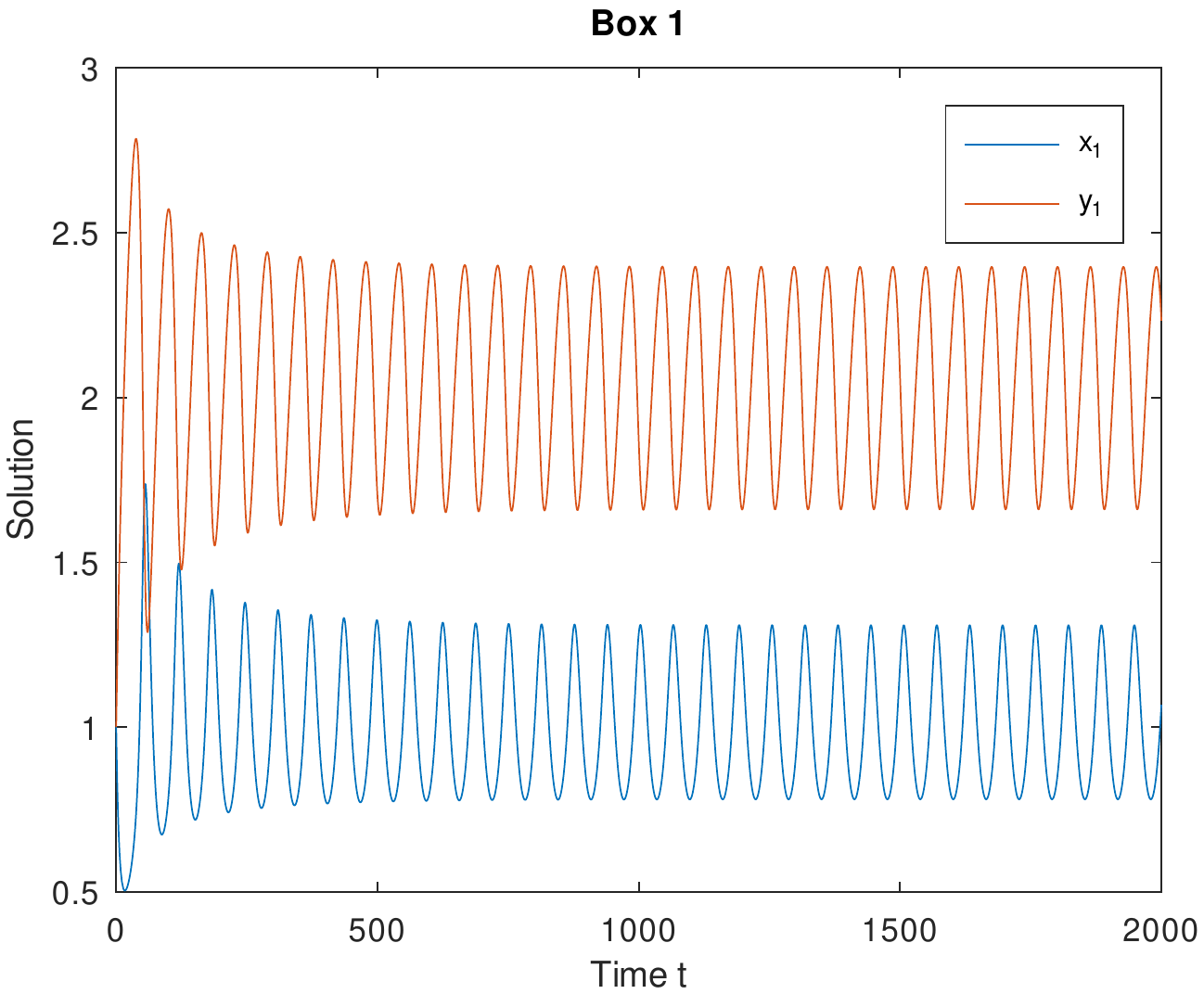}\includegraphics[width=0.4\textwidth]{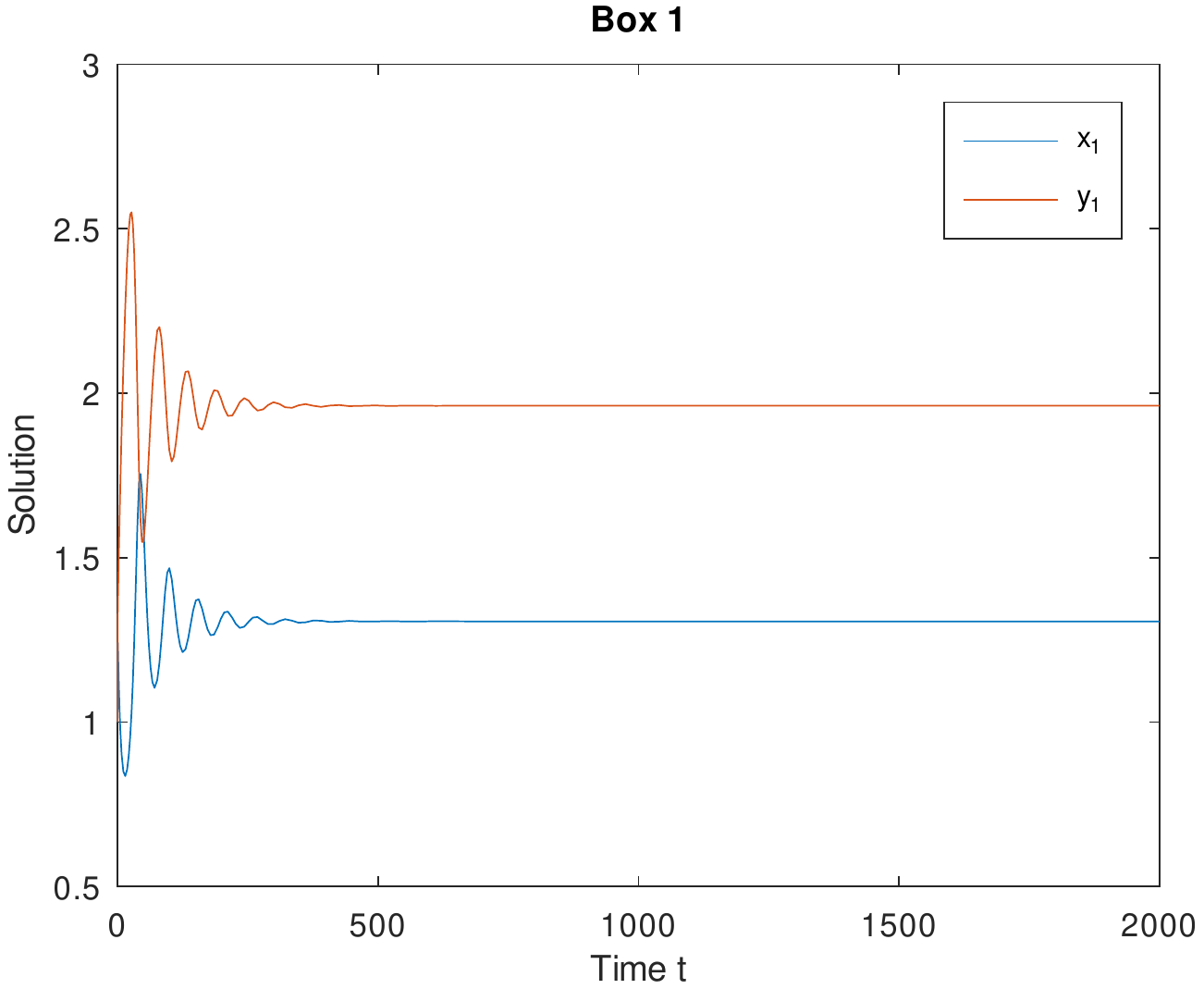}
\caption{The effect of coupling matrices. Here $\beta=2.05$, and we only plot one box for simplicity. (Left): $P=P^L$; (Right): $P=P^{NL}$. \label{fig6}}
\end{figure}

Finally, we provide some numerical simulations to illustrate the theoretical results obtained in Section 3. Here we consider model \eqref{m2}, and choose $n=5$, $\la=0.1$, $\theta=1$ and coupling matrices:
$$P:=\left( {\begin{array}{ccccc}
   -4&1&1&1&1\\
    1&-5&2&1&1\\
    1&1&-5&1&1\\
    1&2&1&-4&2\\
    1&1&1&1&-5
\end{array}} \right),\;\;
Q:=\left( {\begin{array}{ccccc}
   -5&1&1&1&1\\
    1&-6&2&1&1\\
    1&1&-5&1&1\\
    2&1&1&-6&3\\
    1&3&1&3&-6
\end{array}} \right).$$
Moreover, let $\beta$ be the variable parameter and choose the other parameters as Table \ref{table1}.
\begin{table}[htbp]
\setlength{\abovecaptionskip}{0cm}
\setlength{\belowcaptionskip}{0.3cm}
  \centering
  \caption{Initial concentrations of reactants in five boxes.}\label{table1}
    \begin{tabular}{|l|c|c|c|c|c|}
    \hline
  \diagbox{ }{$j$} & 1 & 2 & 3 & 4 & 5 \\
   \hline
    $a_j$ &1&2&1&0.5&1 \\
    $b_j$ &0.1&0.2&0.4&0.1&0.2 \\
    $x_j(0)$ &0.5&1&1&0.5&1 \\
    $y_j(0)$ &2&1&2&1&1 \\
    \hline
    \end{tabular}
\end{table}
We numerically show that model \eqref{m2} can undergoes a Hopf bifurcation, and consequently periodic solutions can arise, see Figs. \ref{fig4} and \ref{fig5}.

\begin{figure}[htbp]
\centering
\includegraphics[width=0.4\textwidth]{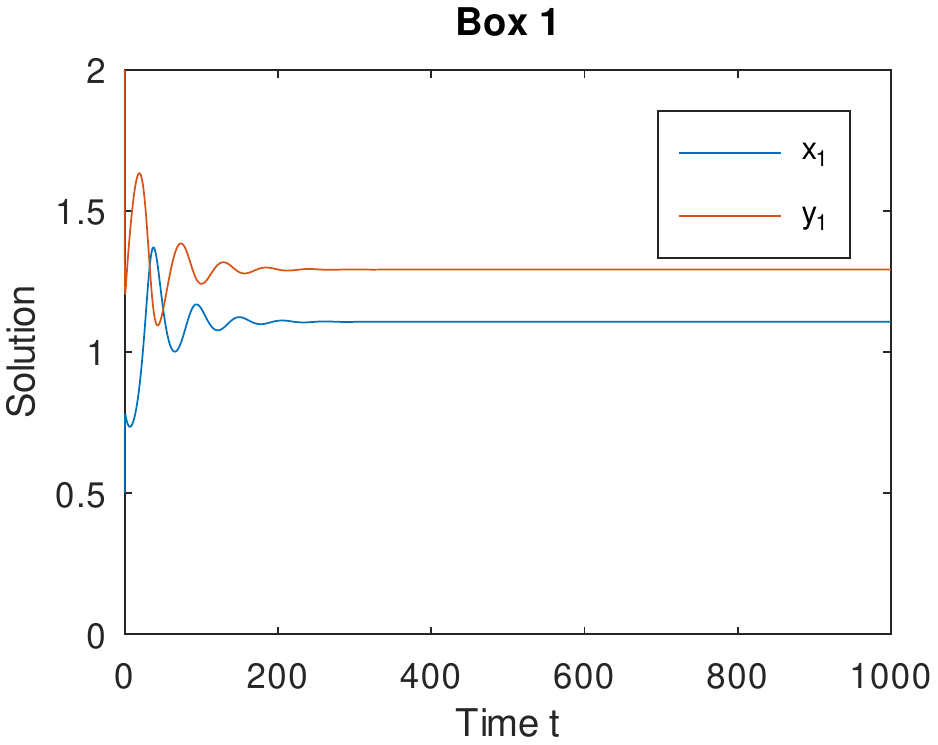}
\includegraphics[width=0.4\textwidth]{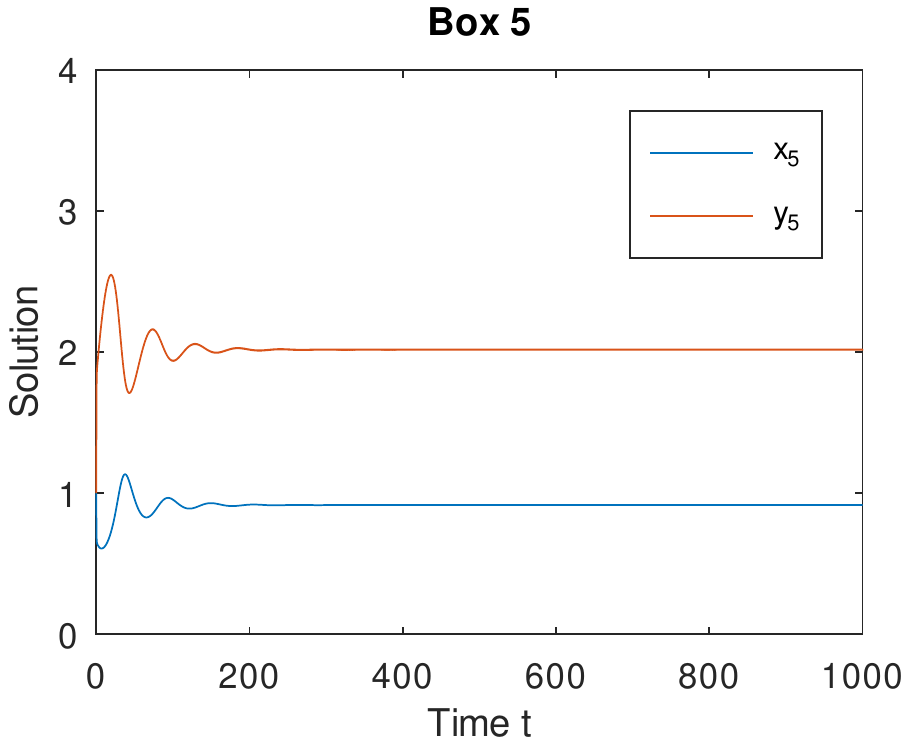}
\caption{The corresponding solution converges to the unique positive equilibrium for model \eqref{m2}. Here $\beta=10$, and we only plot two boxes for simplicity.  \label{fig4}}
\end{figure}

\begin{figure}[htbp]
\centering
\includegraphics[width=0.4\textwidth]{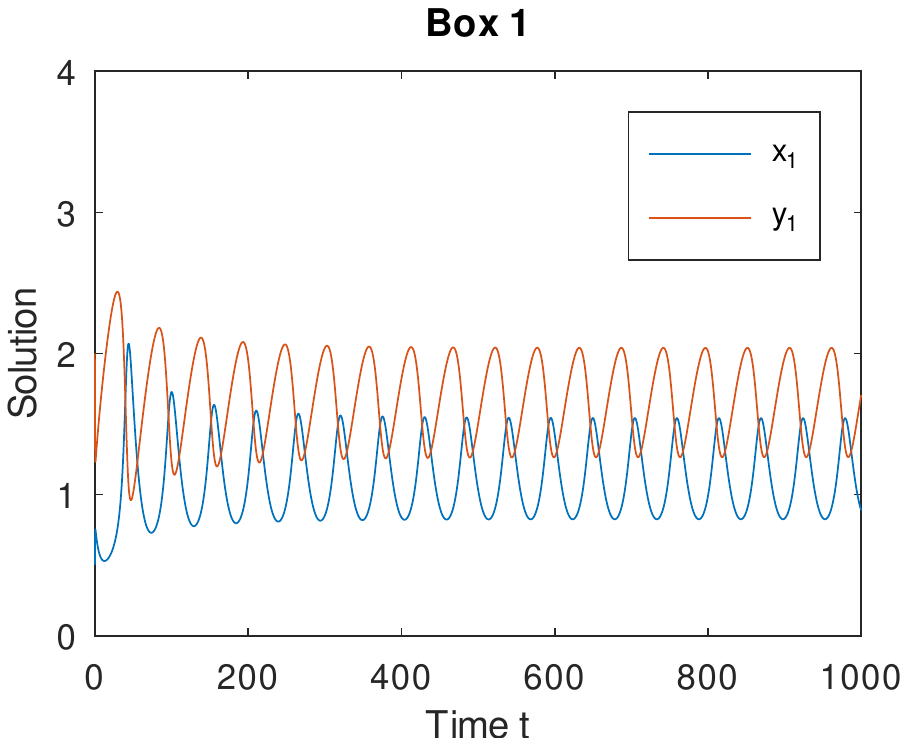}
\includegraphics[width=0.4\textwidth]{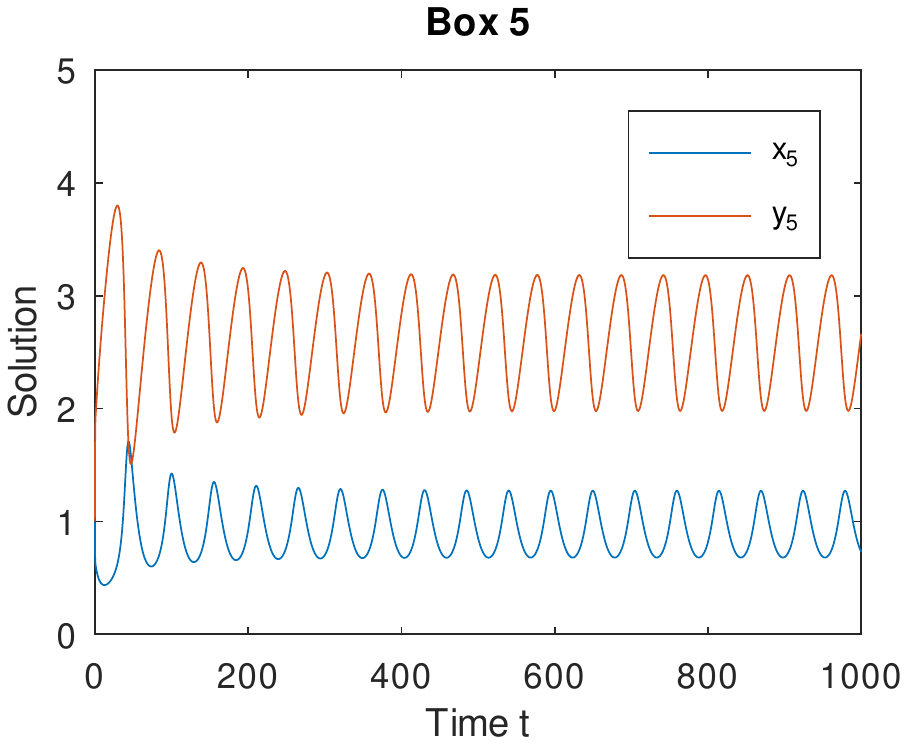}
\caption{The corresponding solution converges to a positive periodic solution for model \eqref{m2}. Here $\beta=13$, and we only plot two boxes for simplicity.\label{fig5}}
\end{figure}

\newpage
\section*{Statements and Declarations}
{\bf {Funding:}} This study was funded by National Natural Science Foundation of China (grant number 12171117) and Shandong Provincial Natural Science Foundation of China (grant number ZR2020YQ01).\\
{\bf{Conflict of Interest:}} The authors declare that they have no conflict of interest.
{\bf{Availability of date and materials:}} All data generated or analysed during this study are included in this published article.







\end{document}